\DeclareSymbolFont{AMSb}{U}{msb}{m}{n}
\documentclass[10pt,a4paper,twoside,leqno,noamsfonts]{amsart}
   \makeatletter 
   \renewcommand\@biblabel[1]{#1.}
   \makeatother

\usepackage[english]{babel}
\usepackage[dvipsnames]{xcolor}
\usepackage{graphicx,pifont}
\usepackage[utopia]{mathdesign}
\usepackage[a4paper,top=2.8cm,bottom=2.8cm,left=2.6cm,
           right=2.6cm,bindingoffset=5mm,marginparwidth=60pt]{geometry}
\usepackage[utf8]{inputenc}
\usepackage{braket,caption,comment,mathtools,stmaryrd}
\usepackage[usestackEOL]{stackengine}
\usepackage{multirow,booktabs,microtype,relsize}
\usepackage[colorlinks,bookmarks]{hyperref} %
      \hypersetup{colorlinks,%
            citecolor=britishracinggreen,%
            filecolor=black,%
            linkcolor=cobalt,%
            urlcolor=black}
      \numberwithin{equation}{section}
\usepackage{ytableau}
\usepackage[colorinlistoftodos]{todonotes}%textwidth=35mm

\definecolor{antiquewhite}{rgb}{0.98, 0.92, 0.84}
\definecolor{buff}{rgb}{0.94, 0.86, 0.51}
\definecolor{palecopper}{rgb}{0.85, 0.54, 0.4}
\definecolor{fluorescentyellow}{rgb}{0.8, 1.0, 0.0}

\definecolor{britishracinggreen}{rgb}{0.0, 0.26, 0.15}
\definecolor{cobalt}{rgb}{0.0, 0.28, 0.67}
\DeclareSymbolFont{usualmathcal}{OMS}{cmsy}{m}{n}
\DeclareSymbolFontAlphabet{\mathcal}{usualmathcal}

\newcommand{\BA}{{\mathbb{A}}}

\newcommand{\BC}{{\mathbb{C}}}

\newcommand{\BE}{{\mathbb{E}}}

\newcommand{\BG}{{\mathbb{G}}}

\newcommand{\BL}{{\mathbb{L}}}

\newcommand{\BN}{{\mathbb{N}}}

\newcommand{\BQ}{{\mathbb{Q}}}
\newcommand{\BR}{{\mathbb{R}}}

\newcommand{\BT}{{\mathbb{T}}}

\newcommand{\BZ}{{\mathbb{Z}}}

\newcommand{\CM}{{\mathcal M}}

\newcommand{\CT}{{\mathcal T}}

\DeclareMathOperator{\points}{points}

\DeclareMathOperator{\Quot}{Quot}

\DeclareMathOperator{\coh}{coh}

\DeclareMathOperator{\Tr}{Tr}

\DeclareMathOperator{\vir}{\mathrm{vir}}

\DeclareMathOperator{\Exp}{Exp}
\DeclareMathOperator{\Var}{Var}

\DeclareMathOperator{\Rep}{R}
\DeclareMathOperator{\GL}{GL}

\DeclareMathOperator{\St}{St}

\newcommand{\dd}{\mathrm{d}}

\newcommand{\crit}{\operatorname{crit}}
\newcommand{\zetass}{\zeta\textrm{-ss}}
\newcommand{\zetast}{\zeta\textrm{-st}}

\DeclareFontFamily{OT1}{rsfs}{}
\DeclareFontShape{OT1}{rsfs}{n}{it}{<-> rsfs10}{}
\DeclareMathAlphabet{\curly}{OT1}{rsfs}{n}{it}

\newcommand\Hom{\operatorname{Hom}}
\newcommand\End{\operatorname{End}}

\newcommand{\DT}{\mathsf{DT}}

\newcommand{\OO}{\mathscr O}

\makeatletter
\newenvironment{proofof}[1]{\par
  \pushQED{\qed}%
  \normalfont \topsep6\p@\@plus6\p@\relax
  \trivlist
  \item[\hskip3\labelsep
        \itshape
    Proof of #1\@addpunct{.}]\ignorespaces
}{%
  \popQED\endtrivlist\@endpefalse
}
\makeatother

\usepackage[all]{xy}
\usepackage{tikz}
\usepackage{tikz-cd}
\usepackage{adjustbox}
\usepackage{rotating}
\usepackage{comment}

\usetikzlibrary{matrix,shapes,intersections,arrows,decorations.pathmorphing}
\tikzset{commutative diagrams/arrow style=math font}
\tikzset{commutative diagrams/.cd,
mysymbol/.style={start anchor=center,end anchor=center,draw=none}}

\tikzset{
shift up/.style={
to path={([yshift=#1]\tikztostart.east) -- ([yshift=#1]\tikztotarget.west) \tikztonodes}
}
}

\theoremstyle{definition}

\newtheorem*{lemma*}{Lemma}
\newtheorem*{theorem*}{Theorem}
\newtheorem*{example*}{Example}
\newtheorem*{fact*}{Fact}
\newtheorem*{notation*}{Notation}
\newtheorem*{definition*}{Definition}
\newtheorem*{prop*}{Proposition}
\newtheorem*{remark*}{Remark}
\newtheorem*{corollary*}{Corollary}

\newtheorem*{conventions*}{Conventions}

\newtheorem{definition}{Definition}[section]

\newtheorem{example}[definition]{Example}

\newtheorem{convention}{Convention}

\newtheorem{remark}[definition]{Remark}

\newtheorem{claim}[definition]{Claim}

\newtheoremstyle{thm} % <name> % (ambienti con dimostrazione)
        {3mm}% <Space above>
        {3mm}% <Space below>
        {\slshape}% <Body font> % 
        {0mm}% <Indent amount>
        {\bfseries}% <Theorem head font>
        {.}% <Punctuation after theorem head>
        {1mm}% <Space after theorem head>
        {}% <Theorem head spec (can be left empty, meaning 'normal')> 
\theoremstyle{thm}

\newtheorem{lemma}[definition]{Lemma}
\newtheorem{prop}[definition]{Proposition}
\newtheorem{thm}{Theorem}

\newtheoremstyle{ex} % <name> % (ambienti con dimostrazione)
        {3mm}% <Space above>
        {3mm}% <Space below>
        {}% <Body font> % \slshape
        {0mm}% <Indent amount>
        {\scshape}% <Theorem head font>
        {.}% <Punctuation after theorem head>
        {1mm}% <Space after theorem head>
        {}% <Theorem head spec (can be left empty, meaning 'normal')> 
\theoremstyle{ex}

\newtheoremstyle{sol} % <name> % (ambienti con dimostrazione)
        {3mm}% <Space above>
        {3mm}% <Space below>
        {}% <Body font> % 
        {0mm}% <Indent amount>
        {\scshape}% <Theorem head font>
        {.}% <Punctuation after theorem head>
        {1mm}% <Space after theorem head>
        {}% <Theorem head spec (can be left empty, meaning 'normal')> 
\theoremstyle{sol}

\providecommand{\bthm}{\begin{thm}}
	\providecommand{\ethm}{\end{thm}}
\providecommand{\bfac}{\begin{fac}}
	\providecommand{\efac}{\end{fac}}
\providecommand{\bprop}{\begin{prop}}
	\providecommand{\eprop}{\end{prop}}
\providecommand{\blem}{\begin{lem}}
	\providecommand{\elem}{\end{lem}}
\providecommand{\bdefn}{\begin{defn}}
	\providecommand{\edefn}{\end{defn}}
\providecommand{\bcor}{\begin{cor}}
	\providecommand{\ecor}{\end{cor}}
\providecommand{\brem}{\begin{rem}}
	\providecommand{\erem}{\end{rem}}
\providecommand{\bproof}{\begin{proof}}
	\providecommand{\eproof}{\end{proof}}
\providecommand{\bnota}{\begin{nota}}
	\providecommand{\enota}{\end{nota}}
\providecommand{\beq}{\begin{equation}}
\providecommand{\eeq}{\end{equation}}
\providecommand{\bexa}{\begin{exa}}
	\providecommand{\eexa}{\end{exa}}
\providecommand{\bclaim}{\begin{claim}}
	\providecommand{\eclaim}{\end{claim}}
\newcounter{x}
\newcounter{y}
\newcounter{z}

\newcommand\xaxis{210}
\newcommand\yaxis{-30}
\newcommand\zaxis{90}

\newcommand\topside[3]{
  \fill[fill=yellow, draw=black,shift={(\xaxis:#1)},shift={(\yaxis:#2)},
  shift={(\zaxis:#3)}] (0,0) -- (30:1) -- (0,1) --(150:1)--(0,0);
}

\newcommand\leftside[3]{
  \fill[fill=green, draw=black,shift={(\xaxis:#1)},shift={(\yaxis:#2)},
  shift={(\zaxis:#3)}] (0,0) -- (0,-1) -- (210:1) --(150:1)--(0,0);
}

\newcommand\rightside[3]{
  \fill[fill=cyan, draw=black,shift={(\xaxis:#1)},shift={(\yaxis:#2)},
  shift={(\zaxis:#3)}] (0,0) -- (30:1) -- (-30:1) --(0,-1)--(0,0);
}

\newcommand\cube[3]{
  \topside{#1}{#2}{#3} \leftside{#1}{#2}{#3} \rightside{#1}{#2}{#3}
}

\newcommand\planepartition[1]{
 \setcounter{x}{-1}
  \foreach \a in {#1} {
    \addtocounter{x}{1}
    \setcounter{y}{-1}
    \foreach \b in \a {
      \addtocounter{y}{1}
      \setcounter{z}{-1}
      \foreach \c in {1,...,\b} {
        \addtocounter{z}{1}
        \cube{\value{x}}{\value{y}}{\value{z}}
      }
    }
  }
}

\title[Higher rank motivic DT invariants of $\BA^3$ via wall-crossing, and asymptotics]{Higher rank motivic  Donaldson--Thomas invariants of $\BA^3$ \\ via wall-crossing, and asymptotics}
\author{Alberto Cazzaniga}
\address{CNR-IOM 
(Trieste, DEMOCRITOS)  
c/o SISSA,
Via Bonomea 265, 34136 Trieste}
\email{cazzaniga@iom.cnr.it}
\author{Dimbinaina Ralaivaosaona}
\address{Department of Mathematical Sciences, Stellenbosch University\newline\indent Private Bag X1, Matieland 7602, South Africa}
\email{naina@sun.ac.za}
\author{Andrea T. Ricolfi}
\address{SISSA, Via Bonomea 265, 34136 Trieste, Italy\newline\indent Institute for Geometry and Physics, via Beirut 4, 34100 Trieste, Italy}
\email{aricolfi@sissa.it}

\begin{document}
\maketitle

\begin{abstract}
We compute, via motivic wall-crossing, the generating function of \emph{virtual motives} of the Quot scheme of points on $\BA^3$, generalising to higher rank a result of Behrend--Bryan--Szendr\H{o}i. We show that this motivic partition function converges to a Gaussian distribution, extending a result of Morrison.
\end{abstract}

{\hypersetup{linkcolor=black}
\tableofcontents}

\section{Introduction}
This paper has a two-fold goal: to compute, and to study the asymptotic behavior of the generating function of rank $r$ motivic Donaldson--Thomas invariants of $\BA^3$, namely the series
\begin{equation*}
\DT_r^{\points}(\BA^3,q) = \sum_{n\geq 0}\,\left[\Quot_{\BA^3}(\mathscr O^{\oplus r},n) \right]_{\vir}\cdot q^n \,\in\,\mathcal M_{\BC}\llbracket q \rrbracket.
\end{equation*}
Here $\mathcal M_{\BC}$ is a suitable motivic ring  and $[\,\cdot\,]_{\vir} \in \mathcal M_{\BC}$ is the \emph{virtual motive} (cf.~\S\,\ref{subsec:motivic_quantum_torus}), induced by the critical locus structure on the Quot scheme $\Quot_{\BA^3}(\mathscr O^{\oplus r},n)$ parametrising $0$-dimensional quotients of length $n$ of the free sheaf $\mathscr O^{\oplus r}$.

The following is our first main result.
\bthm\label{thm:main_motivic} 
There is an identity
\begin{equation}
   \label{formu}
\DT_r^{\points}(\BA^3,q) =
\prod_{m\geq 1}\prod_{k=0}^{rm-1}\left(1-\BL^{2+k-\frac{rm}{2}}q^{m}\right)^{-1}.
\end{equation}
Moreover, this series factors as $r$ copies of shifted rank $1$ contributions: there is an identity
\begin{equation}
   \label{formula_product}
\DT_r^{\points}(\BA^3,q)=\prod_{i=1}^r \DT_1^{\points}\left(\BA^3,q\BL^{\frac{-r-1}{2}+i}\right).
\end{equation}
\ethm
The result was first obtained in the case $r=1$ by Behrend, Bryan and Szendr\H{o}i \cite{BBS} via an explicit motivic vanishing cycle calculation. Formula \eqref{formula_product} follows by combining Formula \eqref{formu} and Lemma \ref{thm:quot_partition_function}. The approach of \S\,\ref{subsec:calculation_framed_3_loop_quiver}, where we prove Formula \eqref{formu}, is based on the techniques of \emph{motivic wall-crossing} for framed objects developed by Mozgovoy \cite{Mozgovoy_Framed_WC}, allowing us to express the invariants for $\Quot_{\BA^3}(\mathscr O^{\oplus r},n)$, which we view as `$r$-framed' Donaldson--Thomas invariants, in terms of the universal series of the invariants of \emph{unframed} representations of the $3$-loop quiver in a critical chamber. These ideas can be employed to compute framed motivic Donaldson--Thomas invariants of small crepant resolutions of affine toric Calabi–Yau 3-folds \cite{Cazza_Ric}, which also exhibit similar factorisation properties.

The fact that partition functions of rank $r$ invariants factor as $r$ copies of partition functions of rank $1$ invariants, shifted just as in Formula \eqref{formula_product}, has also been observed in the context of K-theoretic Donaldson--Thomas theory of $\BA^3$  \cite{FMR_K-DT}, as well as in string theory \cite{Magnificent_colors}. The exponential form of Formula \eqref{formu} has been exploited in \cite{Quot19} to define higher rank motivic Donaldson--Thomas invariants for an arbitrary smooth quasi-projective $3$-fold. 

\smallbreak 
Formula \eqref{formu} allows us to interpret the refined Donaldson--Thomas invariants of $\Quot_{\BA^3}(\mathscr O^{\oplus r},n)$ in terms of a weighted count of $r$-tuples of plane partitions $\overline{\pi}=(\pi_{1}, \ldots, \pi_{r})$ of total size $n$ (also known in the physics literature as $r$-\emph{colored plane partitions}). Setting $T=\BL^{1/2}$, the coefficient of $q^{n}$ in $\DT_r^{\points}(\BA^3,q)$ can be written as 
\begin{equation}\label{eq:coefM}
M_{n,r}(T)=\sum_{\overline{\pi}} T^{\,S_{n,r}(\overline{\pi})},    
\end{equation}
where $S_{n,r}$ is a certain explicit random variable on the space of $r$-tuples of plane partitions. In \S\,\ref{sec:Asymptotics}, we describe the asymptotic behavior of (a renormalisation of) the refined DT generating series, generalising A.~Morrison's result for $r=1$ \cite{Morrison_asymptotics}. We discuss the relationship with Morrison's work in \S\,\ref{sec:random_variables_on_colored_partitions}.

The following is our second main result. It will be proved in \S\,\ref{sub:proofB}.

\bthm\label{main1} As $n\to \infty$, the normalised random variable $n^{-2/3}S_{n,r}$ converges in distribution to $\mathcal{N}(\mu, \sigma^2)$ with
\[
\mu = \frac{r^{1/3}\pi^2}{2^{5/3}(\zeta(3))^{2/3}}\, \text{ and }\, \sigma^2=\frac{r^{5/3}}{(2\zeta(3))^{1/3}},
\]
where $\zeta(s)$ is Riemann's zeta function.
\ethm

\subsection*{Acknowledgments}
The very existence of this work owes a debt to Bal\'{a}zs Szendr\H{o}i, who brought the authors together. We thank him for his support and for generously sharing his ideas throughout the years. A.C.~thanks CNR-IOM for support and the excellent working conditions, and the African Institute for Mathematical Sciences (AIMS), South Africa, for support during the first part of this collaboration. D.R.~is supported by Division for Research Development (DRD) of
Stellenbosch University and the National Research Foundation (NRF) of South Africa. A.R.~is supported by Dipartimenti di Eccellenza and thanks SISSA for the excellent working conditions.

%%%%%%%%%%%%%%%%%%%%%%%%%%%%%%%%%%%%%%%%%%%%%%%%%%%%%%%%%%%%%%%%
%%%%%%%%%%%%%%%%%%%%%%%%%%%%%%%%%%%%%%%%%%%%%%%%%%%%%%%%%%%%%%%%
\section{Background material}
\label{sec:background_material}

%%%%%%%%%%%%%%%%%%%%%%%%%%%%%%%%%%%%%%%%%%%%%%%%%%%%%%%%%%%%%%%%
\subsection{Rings of motives and the motivic quantum torus}
\label{subsec:motivic_quantum_torus}

Let $K_0(\St_{\BC})$ be the Grothendieck ring of stacks. It can be defined as the localisation of the ordinary Grothendieck ring of varieties $K_0(\Var_{\BC})$ at the classes $[\GL_k]$ of general linear groups \cite{Bri-Hall}. The invariants we want to study will live in the extended ring 
\[
\CM_{\BC} = K_0(\St_{\BC})\bigl[\BL^{-\frac{1}{2}}\bigr],
\]
where $\BL = [\BA^1] \in K_0(\Var_{\BC}) \to K_0(\St_{\BC})$ is the Lefschetz motive. 

%%%%%%%%%%%%%%%%%%%%%%%%%%%%%%%%%%%%%%%%%%%%%%%%%%%%%%%%%%%%%%
\subsubsection{The virtual motive of a critical locus}\label{subsec:virtual_motive}

Let $U$ be a smooth $d$-dimensional $\BC$-scheme, $f\colon U \to \BA^1$ a regular function. The \emph{virtual motive} of the critical locus $\crit f = Z( \dd f) \subset U$, depending on the pair $(U,f)$, is defined in \cite{BBS} as the motivic class
\[
\bigl[\crit f\bigr]_{\vir} = -\BL^{-\frac{d}{2}}\cdot \left[\phi_f\right] \,\in \,\CM_{\BC}^{\hat\mu},
\]
where $[\phi_f] \in K_0^{\hat\mu}(\Var_{\BC})$ is the (absolute) motivic vanishing cycle class defined by Denef and Loeser \cite{DenefLoeser1}. The `$\hat\mu$' decoration means that we are considering $\hat\mu$-equivariant motives, where $\hat\mu$ is the group of all roots of unity. However, the motivic invariants studied here will live in the subring $\CM_{\BC}\subset \CM_{\BC}^{\hat\mu}$ of classes carrying the trivial action. 

\begin{example}\label{example:vir_motive_smooth_scheme}
Set $f = 0$. Then $\crit f = U$ and  $[U]_{\vir} = \BL^{-(\dim U)/2}\cdot [U]$. For instance, $[\GL_k]_{\vir} = \BL^{-k^2/2}\cdot [\GL_k]$.
\end{example}

\begin{remark}
We use lambda-ring conventions on $\mathcal M_{\BC}$ from \cite{BBS,DavisonR}. In particular we use the definition of $[\crit f]_{\vir}$ from \cite[\S\,2.8]{BBS}, which differs (slightly) from the one in \cite{RefConifold}. The difference amounts to the substitution $\BL^{1/2}\to -\BL^{1/2}$. The Euler number specialisation with our conventions is $\BL^{1/2}\to -1$.
\end{remark}

%%%%%%%%%%%%%%%%%%%%%%%%%%%%%%%%%
\subsection{Quivers and motivic quantum torus}
A quiver $Q$ is a finite directed graph, determined by its sets $Q_0$ and $Q_1$ of vertices and edges, respectively, along with the maps $h$, $t\colon Q_1 \to Q_0$ specifying where an edge starts or ends.
We use the notation
\[
\begin{tikzcd}
t(a) \,\,\bullet  \arrow{rr}{a} & & \bullet\,\, h(a)
\end{tikzcd}
\]
to denote the \emph{tail} and the \emph{head} of an edge $a \in Q_1$.

All quivers in this paper will be assumed connected.
The \emph{path algebra} $\BC Q$ of a quiver $Q$ is defined, as a $\BC$-vector space, by using as a $\BC$-basis the set of all paths in the quiver, including a trivial path $\epsilon_i$ for each $i \in Q_0$. The product is defined by concatenation of paths whenever the operation is possible, and is set to be $0$ otherwise. The identity element is $\sum_{i\in Q_0}\epsilon_i \in \BC Q$.

On a quiver $Q$ one can define the \emph{Euler--Ringel form} $\chi_Q(-,-)\colon \BZ^{Q_0}\times \BZ^{Q_0} \to \BZ$ by
\[
\chi_Q(\alpha,\beta) = \sum_{i \in Q_0}\alpha_i\beta_i - \sum_{a \in Q_1}\alpha_{t(a)}\beta_{h(a)},
\]
as well as the skew-symmetric form
\[
\braket{\alpha,\beta}_Q = \chi_Q(\alpha,\beta)-\chi_Q(\beta,\alpha).
\]

\begin{definition}[$r$-framing]\label{def:r-framing}
Let $Q$ be a quiver with a distinguished vertex $0\in Q_0$, and let $r$ be a positive integer. We define the quiver $\widetilde Q$ by adding one vertex, labelled $\infty$, to the original vertices in $Q_0$, and $r$ edges $\infty\to 0$. We refer to $\widetilde Q$ as the $r$-\emph{framed} quiver obtained out of $(Q,0)$.
\end{definition}

Let $Q$ be a quiver. Define its \emph{motivic quantum torus} (or \emph{twisted motivic algebra}) as
\[
\mathcal T_Q = \prod_{\alpha \in \BN^{Q_0}} \mathcal{M}_{\BC}\cdot y^\alpha
\]
with product
\begin{equation}\label{eqn:product_in_TQ}
    y^\alpha\cdot y^\beta = \BL^{\frac{1}{2}\braket{\alpha,\beta}_Q}y^{\alpha+\beta}.
\end{equation}
If $\widetilde{Q}$ is the $r$-framed quiver associated to $(Q,0)$, one has a decomposition 
\[
\mathcal T_{\widetilde{Q}} = \mathcal T_Q \oplus \prod_{d\geq 0} \CM_{\BC}\cdot y_\infty^d,
\]
where we have set $y_\infty = y^{(1,\mathbf 0)}$. A generator $y^\alpha \in \mathcal T_Q$ will be identified with its image $y^{(0,\alpha)} \in \mathcal T_{\widetilde{Q}}$.

%%%%%%%%%%%%%%%%%%%%%%%%%%%%%%%%%%%%%%%%%%%%%%%%%%%%%%%%%%%%%%%%%%
\subsection{Quiver representations and their stability}

Let $Q$ be a quiver.
A \emph{representation} $\rho$ of $Q$ is the datum of a finite dimensional $\BC$-vector space $\rho_i$ for every vertex $i\in Q_0$, and a linear map $\rho(a)\colon \rho_i\to \rho_j$ for every edge $a\colon i\to j$ in $Q_1$.
The \emph{dimension vector} of $\rho$ is $
\underline{\dim}\,\rho = (\dim_{\BC} \rho_i)_i\in \mathbb N^{Q_0}$, where $\BN=\BZ_{\geq 0}$.

\begin{convention}\label{order_of_dimensions}
Let $Q$ be a quiver, $\widetilde Q$ its $r$-framing. The dimension vector of a representation $\widetilde \rho$ of $\widetilde Q$ will be denoted by $(d,\alpha)$, where $d = \dim_{\BC}\widetilde{\rho}_\infty \in \BN$ and $\alpha \in \BN^{Q_0}$.
\end{convention}

The space of all representations of $Q$ with a fixed dimension vector $\alpha\in \mathbb N^{Q_0}$ is the affine space
\[
\Rep(Q,\alpha) = \prod_{a \in Q_1}\Hom_{\BC}(\BC^{\alpha_{t(a)}},\BC^{\alpha_{h(a)}}).
\]
The gauge group $\GL_\alpha = \prod_{i\in Q_0} \GL_{\alpha_i}$ acts on $\Rep(Q,\alpha)$ by $(g_i)_i \cdot (\rho(a))_{a\in Q_1} = (g_{h(a)}\circ\rho(a)\circ g_{t(a)}^{-1})_{a \in Q_1}$.

Following \cite{RefConifold}, we recall the notion of (semi)stability of a representation.

\begin{definition}\label{centralcharge}
A \emph{central charge} is a group homomorphism $\mathrm{Z}\colon \mathbb Z^{Q_0}\to \BC$ such that the image of $\mathbb N^{Q_0}\setminus 0$ lies inside $\mathbb H_+ = \set{re^{\sqrt{-1}\pi\varphi}|r>0,\,0<\varphi\leq 1}$. For every $\alpha\in \mathbb N^{Q_0}\setminus 0$, we denote by $\varphi(\alpha)$ the real number $\varphi$ such that $\mathrm{Z}(\alpha) = re^{\sqrt{-1}\pi\varphi}$. It is called the \emph{phase} of $\alpha$ with respect to $\mathrm{Z}$.
\end{definition}

Note that every vector $\zeta\in \BR^{Q_0}$ induces a central charge $\mathrm{Z}_{\zeta}$ if we set
$\mathrm{Z}_{\zeta}(\alpha) = -\zeta\cdot \alpha + \lvert\alpha\rvert\sqrt{-1}$,
where $\lvert\alpha\rvert = \sum_{i \in Q_0}\alpha_i$. We denote by $\varphi_\zeta$ the induced phase function, and we set $
\varphi_\zeta(\rho) = \varphi_\zeta(\underline{\dim}\,\rho)$ for every representation $\rho$ of $Q$.

\begin{definition}\label{stablereps}
Fix $\zeta\in \BR^{Q_0}$. A representation $\rho$ of $Q$ is called \emph{$\zeta$-semistable} if 
\[
\varphi_\zeta(\rho')\leq \varphi_\zeta(\rho)
\]
for every nonzero proper subrepresentation $0\neq \rho'\subsetneq \rho$. If `$\leq$' can be replaced by `$<$', we say that $\rho$ is \emph{$\zeta$-stable}. Vectors $\zeta\in \BR^{Q_0}$ are referred to as \emph{stability parameters}. 
\end{definition}

\begin{definition}
Let $\alpha \in \BN^{Q_0}$ be a dimension vector. A stability parameter $\zeta$ is called $\alpha$-\emph{generic} if for any $0<\beta<\alpha$ one has $\varphi_\zeta(\beta) \neq \varphi_\zeta(\alpha)$. This implies that every $\zeta$-semistable representation of $Q$, of dimension $\alpha$, is $\zeta$-stable.
\end{definition}

The sets of $\zeta$-stable and $\zeta$-semistable representations with given dimension vector $\alpha$ form a chain of open subsets 
\[
\Rep^{\zetast}(Q,\alpha)\subset \Rep^{\zetass}(Q,\alpha)\subset \Rep(Q,\alpha).
\]

%%%%%%%%%%%%%%%%%%%%%%%%%%%%%%%%%
\subsection{Quivers with potential}
\label{subsec:quiver_with_potential}
Let $Q$ be a quiver. Consider the quotient $\BC Q / [\BC Q,\BC Q]$ of the path algebra by the commutator ideal. A finite linear combination of cyclic paths $W \in \BC Q / [\BC Q,\BC Q]$ 
is called a \emph{superpotential}. 
Given a cyclic path $w$ and an arrow $a \in Q_1$, one defines the noncommutative derivative
\[
\frac{\partial w}{\partial a} = 
\sum_{\substack{w=cac' \\ c,c'\textrm{ paths in }Q}}c'c\,\in\,\BC Q.
\]
This rule extends to an operator $\partial/\partial a$ acting on every superpotential. The \emph{Jacobi algebra} $J=J_{Q,W}$ of $(Q,W)$ is the quotient of $\BC Q$ by the two-sided ideal generated by $\partial W/\partial a$ for all edges $a \in Q_1$.
For every $\alpha \in \BN^{Q_0}$, the superpotential $W = \sum_ca_cc$ determines a regular function
\[
f_\alpha \colon \Rep(Q,\alpha) \to \BA^1,\quad \rho \mapsto \sum_{c\textrm{ cycle in }Q}a_c \Tr (\rho(c)).
\]
The points in the critical locus $\crit f_\alpha \subset \Rep(Q,\alpha)$ correspond to $\alpha$-dimensional $J$-\emph{modules}. 

Fix an $\alpha$-generic stability parameter $\zeta \in \BR^{Q_0}$. If $f_{\zeta,\alpha} \colon \Rep^{\zetast}(Q,\alpha)\to \BA^1$ is the restriction of $f_\alpha$, then
\[
\mathfrak M(J,\alpha) = [\crit f_\alpha/G_\alpha],\quad \mathfrak M_\zeta(J,\alpha) = [\crit f_{\zeta,\alpha} / \GL_\alpha]
\]
are, by definition, the stacks of $\alpha$-dimensional $J$-modules and $\zeta$-stable $J$-modules.  

\begin{definition}[\cite{RefConifold}]
We define motivic Donaldson--Thomas invariants
\begin{equation}
 \label{eqn:motivicDT}
    \bigl[\mathfrak M(J,\alpha)\bigr]_{\vir} 
    \,=\, \frac{\left[\crit f_\alpha\right]_{\vir}}{\left[\GL_\alpha\right]_{\vir}},\qquad
    \left[\mathfrak M_\zeta(J,\alpha)\right]_{\vir} 
    \,=\, \frac{\left[\crit f_{\zeta,\alpha}\right]_{\vir}}{\left[\GL_\alpha\right]_{\vir}}
\end{equation}
in $\CM_{\BC}$, where $[\GL_\alpha]_{\vir}$ is taken as in Example \ref{example:vir_motive_smooth_scheme}.
The generating function
\begin{equation}\label{definition_AU}
A_U = \sum_{\alpha \in \BN^{Q_0}}\,\bigl[\mathfrak M(J,\alpha)\bigr]_{\vir}\cdot y^\alpha \,\in\,\CT_{Q}
\end{equation}
is called the \emph{universal series} of $(Q,W)$.
\end{definition}

\begin{definition}
A stability parameter $\zeta \in \BR^{Q_0}$ is called \emph{generic} if $\zeta\cdot \underline{\dim}\,\rho \neq 0$ for every nontrivial $\zeta$-stable $J$-module $\rho$.
\end{definition}

%%%%%%%%%%%%%%%%%%%%%%%%%%%%%%%%%%%%%%%%%%%%%%%%%
\subsection{Framed motivic DT invariants}
\label{subsec:motivicDT_quiver_potential}
Let $Q$ be a quiver, $r\geq 1$ be an integer, and consider its $r$-framing $\widetilde Q$ with respect to a vertex $0 \in Q_0$ (Definition \ref{def:r-framing}). 
A representation $\widetilde \rho$ of $\widetilde Q$ can be uniquely written as a pair $(u,\rho)$, where $\rho$ is a representation of $Q$ and $u = (u_1,\dots,u_r)$ is an $r$-tuple of linear maps $u_i\colon \widetilde{\rho}_\infty\to \rho_0$. From now on, we assume our framed representations to satisfy $\dim_{\BC} \widetilde{\rho}_\infty = 1$, so that according to Convention \ref{order_of_dimensions} we can write $\underline{\dim}\,\widetilde \rho=(1,\underline{\dim}\,\rho)$. We also view $\rho$ as a subrepresentation of $\widetilde{\rho}$ of dimension $(0,\underline{\dim}\,\rho)$, based at the vertex $0 \in Q_0$.

\begin{definition}\label{def:framedstability}
Fix $\zeta\in \mathbb R^{Q_0}$.
A representation $(u,\rho)$ of $\widetilde Q$ (resp.~a $\widetilde J$-module) with $\dim_{\BC} \widetilde{\rho}_\infty = 1$ is said to be \emph{$\zeta$-(semi)stable} if it is $(\zeta_\infty,\zeta)$-(semi)stable in the sense of Definition \ref{stablereps}, where $\zeta_\infty = -\zeta\cdot \underline{\dim}\,\rho$.
\end{definition}

We now define motivic DT invariants for moduli stacks of $r$-framed representations of a given quiver $Q$. Fix a superpotential $W$ on $Q$. Let $\widetilde{Q}$ be the $r$-framing of $Q$ at a given vertex $0 \in Q_0$, and let $\widetilde J$ be the Jacobi algebra $J_{\widetilde{Q},W}$, where $W$ is viewed as a superpotential on $\widetilde Q$ in the obvious way. For a generic stability parameter $\zeta \in \BR^{Q_0}$, and an arbitrary dimension vector $\alpha \in \BN^{Q_0}$, set
\[
\zeta_\infty = -\zeta\cdot \alpha,\quad \widetilde \zeta = (\zeta_\infty,\zeta),\quad \widetilde \alpha = (1,\alpha).
\]
As in \S\,\ref{subsec:quiver_with_potential}, consider the trace map  $f_{\widetilde{\alpha}}\colon\Rep(\widetilde{Q},\widetilde{\alpha})\to \BA^1$,
induced by $W$, and its restriction to the framed-stable locus $f_{\widetilde{\zeta},\widetilde{\alpha}}\colon\Rep^{\widetilde{\zeta}\textrm{-st}}(\widetilde{Q},\widetilde{\alpha}) \to \BA^1$. Define the moduli stacks
\[
\mathfrak M(\widetilde J,\alpha) = \left[\crit f_{\widetilde{\alpha}} \,\big/ \GL_\alpha\right],
\quad \mathfrak M_\zeta (\widetilde J,\alpha) = \left[\crit f_{\widetilde{\zeta},\widetilde{\alpha}} \,\big/ \GL_\alpha\right].
\]
Note that we are not quotienting by $\GL_{\widetilde{\alpha}} = \GL_\alpha \times \BC^\times$, but only by $\GL_\alpha$.

\begin{definition}\label{def:motivic_partition_functions}
We define $r$-framed motivic Donaldson--Thomas invariants
\[
    \left[\mathfrak M(\widetilde J,\alpha) \right]_{\vir}
    \,=\,\frac{\left[\crit f_{\widetilde\alpha}\right]_{\vir}}{\left[\GL_{\alpha}\right]_{\vir}},\qquad
    \left[\mathfrak M_\zeta(\widetilde J,\alpha) \right]_{\vir}
    \,=\,\frac{\bigl[\crit f_{\widetilde \zeta,\widetilde\alpha}\bigr]_{\vir}}{\left[\GL_{\alpha}\right]_{\vir}}
\]
and the associated motivic generating functions
\begin{align*}
\widetilde A_U 
&= \sum_{\alpha \in \BN^{Q_0}} \,\left[\mathfrak M(\widetilde J,\alpha) \right]_{\vir}\cdot y^{\widetilde\alpha} \in \mathcal T_{\widetilde Q} \\
\mathsf Z_\zeta 
&= \sum_{\alpha \in \BN^{Q_0}} \,\left[\mathfrak M_\zeta(\widetilde J,\alpha) \right]_{\vir}\cdot y^{\widetilde\alpha} \in \mathcal T_{\widetilde Q}.
\end{align*}
\end{definition}

%%%%%%%%%%%%%%%%%%%%%%%%%%%%%%%%%%%%%%%
\subsection{Dimensional reduction}
We say that a quiver with potential $(Q,W)$ admits a \emph{cut} if there is a subset $I \subset Q_1$ such that every cyclic monomial appearing in $W$ contains exactly one edge in $I$.

If $I$ is a cut for $(Q,W)$, one can define a new quiver $Q_I = (Q_0,Q_1\setminus I)$. Let $J_{W,I}$ be the quotient of $\BC Q_I$ by the two-sided ideal generated by the noncommutative derivatives $\partial W/\partial a$ for $a \in I$.
Let $\Rep(J_{W,I},\alpha) \subset \Rep(Q_I,\alpha)$ be the space of $J_{W,I}$-modules of dimension vector $\alpha$. Then one has the following dimensional reduction principle.

\begin{prop}[{\cite[Prop.~1.15]{RefConifold}}]
\label{prop:cut}
Suppose $I$ is a cut for $(Q,W)$. Set $\dd_I(\alpha) = \sum_{a \in I}\alpha_{t(a)}\alpha_{h(a)}$. Then
\[
A_U = \sum_{\alpha \in \BN^{Q_0}}\BL^{\frac{1}{2}\chi_Q(\alpha,\alpha) + \dd_I(\alpha)}\frac{[\Rep(J_{W,I},\alpha)]}{[\GL_\alpha]} \cdot y^\alpha.
\]
\end{prop}

\begin{example}\label{example:cut_3_loop_quiver}
Let $Q=L_3$ be the $3$-loop quiver (see Figure \ref{fig:3loopquiver_framed}, and remove the framing vertex to obtain a picture of this quiver) with the potential $W=A_3[A_1,A_2]$. Notice that $J=J_{L_3,W}=\BC[x,y,z]$, and $I = \set{A_3}$ is a cut for $(L_3,W)$. The quiver $Q_I$ is the $2$-loop quiver and $J_{W,I} = \BC[x,y]$. We have $\dd_I(n)=n^2$ and $\chi_Q(n,n)=-2n^2$. Therefore Proposition \ref{prop:cut} yields an identity
\begin{equation}\label{eqn:cut_yields_FF}
\sum_{n\geq 0}\,\bigl[\mathfrak M(J,n) \bigr]_{\vir}\cdot y^n = \sum_{n\geq 0}\frac{[C_n]}{[\GL_n]}\cdot y^n = \prod_{m \geq 1}\prod_{k \geq 1}\left(1-\BL^{2-k}y^m\right)^{-1},
\end{equation}
where $\Rep(J_{W,I},n)$ is identified with the \emph{commuting variety}
\[
C_n = \Set{(A_1,A_2) \in \End_{\BC}(\BC^n)^{\oplus 2} | [A_1,A_2] = 0} \subset \End_{\BC}(\BC^n)^{\oplus 2},
\]
and the second identity in \eqref{eqn:cut_yields_FF} is the Feit--Fine formula \cite{FF1,BBS,BM15}.
\end{example}

\begin{remark}
The universal series $A_{U}$ has been computed for several homogeneous deformations of the potential $W$ of Example \ref{example:cut_3_loop_quiver} in \cite{DEF-MOT-DT}.
\end{remark}

%%%%%%%%%%%%%%%%%%%%%%%%%%%%%%%%%%%%%%%%%%%%%%%%%%%%%%%%%%%%%%%%%%
\section{Motivic DT invariants of the Quot scheme of points}
\label{sec:quot_scheme_of_points_via_WC}
%%%%%%%%%%%%%%%%%%%%%%%%%%%%%%%%%%%%%%%%%%%%%%%%%%%%%%%%%%%%%%%%%%
\subsection{Stability on the framed 3-loop quiver}
\label{subsec:framed_3_loop_and_its_stab}
The main character in this section is the framed quiver $\widetilde{L}_3$ of Figure \ref{fig:3loopquiver_framed}, which we equip with the superpotential $W=A_3[A_1,A_2]$.

\begin{figure}[ht]
\begin{tikzpicture}[>=stealth,->,shorten >=2pt,looseness=.5,auto]
  \matrix [matrix of math nodes,
           column sep={3cm,between origins},
           row sep={3cm,between origins},
           nodes={circle, draw, minimum size=7.5mm}]
{ 
|(A)| \infty & |(B)| 0 \\         
};
\tikzstyle{every node}=[font=\small\itshape]
\path[->] (B) edge [loop above] node {$A_1$} ()
              edge [loop right] node {$A_2$} ()
              edge [loop below] node {$A_3$} ();

\node [anchor=west,right] at (-0.15,0.11) {$\vdots$};

\draw (A) to [bend left=25,looseness=1] (B) node [midway,above] {};
\draw (A) to [bend left=40,looseness=1] (B) node [midway] {};
\draw (A) to [bend right=35,looseness=1] (B) node [midway,below] {};
\end{tikzpicture}
\caption{The $r$-framed $3$-loop quiver $\widetilde{L}_3$.}\label{fig:3loopquiver_framed}
\end{figure}
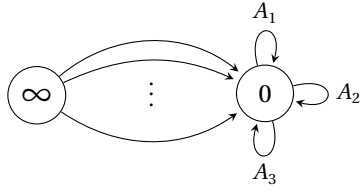

\begin{definition}\label{def:span}
Let $\widetilde{\rho}=(u,\rho)$ be a representation of $\widetilde{L}_3$ of dimension $(1,n)$. We denote by $\braket{u,\rho} \subset \widetilde{\rho}$ the smallest subrepresentation of $\widetilde{\rho}$ containing $u(\rho_{\infty})$. More precisely, if $\rho = (A_1,A_2,A_3) \in \Rep(L_3,n)$, then $\braket{u,\rho}$ is the subrepresentation of $\widetilde\rho$ with $\braket{u,\rho}_\infty = \widetilde{\rho}_\infty = \BC$ and
\[
\braket{u,\rho}_0 = \textrm{span}_{\BC}\Set{A_1^{a_1}A_2^{a_2}A_3^{a_3}\cdot u_\ell(1)|a_i\geq 0,1\leq \ell\leq r} \subset \widetilde{\rho}_0,
\]
and with linear maps induced naturally by those defined by $\widetilde{\rho}$.
\end{definition}

From now on we identify the space of stability parameters for $L_3$ with $\BR$.

\begin{lemma}\label{lemma:stability}
Let $\zeta\in \BR$ be a stability parameter, and let $\widetilde{\rho} = (u,\rho)$ be a representation of  $ \  \widetilde{L}_3$ of dimension $(1,n)$. Set $\widetilde \zeta = (-n\zeta,\zeta)$. Then:
\begin{enumerate}
    \item if $\zeta<0$, $\widetilde{\rho}$ is $\zeta$-semistable if and only if it is $\zeta$-stable if and only if $\widetilde{\rho}=\braket{u,\rho}$;\label{item11}
    \item if $\zeta=0$, $\widetilde{\rho}$ is $\zeta$-semistable;\label{item22}
    \item if $\zeta>0$, $\widetilde{\rho}$ is $\zeta$-semistable if and only if it is $\zeta$-stable if and only if $n=0$.\label{item33}
\end{enumerate}
\end{lemma}
\begin{proof}
For the case $\zeta<0$ we refer to \cite[Prop. 2.4]{BR18}. Consider the case $\zeta>0$. If we had $n = \dim_{\BC}\widetilde{\rho}_0 > 0$, then $\rho \subset \widetilde{\rho}$ would be destabilising, for $\mathrm{Z}_{\widetilde{\zeta}}(0,n) = -n\zeta + n\sqrt{-1}$ implies $\varphi_{\widetilde{\zeta}}(\rho) > 1/2 = \varphi_{\widetilde{\zeta}}(\widetilde{\rho})$, since $\mathrm{Z}_{\widetilde{\zeta}}(1,n)=(n+1)\sqrt{-1}$ has vanishing real part.
On the other hand, if $n = 0$ then $\widetilde{\rho}$ is simple and hence $\zeta$-stable.
In the case $\zeta=0$ there is nothing to prove, as all representations have phase $1/2$. 
\end{proof}

Consider the following regions of the space of stability parameters $\BR$:
\begin{itemize}
    \item [$\circ$] $\Omega_{+} = \set{\zeta\in \BR | \zeta < 0 }$,
    \item [$\circ$] $\Omega_{0} = \set{\zeta\in \BR | \zeta=0 }$,
    \item [$\circ$] $\Omega_{-} = \set{\zeta\in \BR | \zeta > 0}$.
\end{itemize}
By Lemma \ref{lemma:stability} the space of stability parameters on $\widetilde{{L}}_3$ admits a particularly simple wall-and-chamber decomposition $\BR = \Omega_{+} \amalg \Omega_{0} \amalg \Omega_{-}$ with one wall (the origin) and two chambers. 

%%%%%%%%%%%%%%%%%%%%%%%%%%%%%%%%%%%%%%%%%%%%%%%%%%%%%%%%%%%%%%%%%%
\subsection{The virtual motive of the Quot scheme of points}
Let $\widetilde{L}_3$ be the $r$-framed $3$-loop quiver (Figure \ref{fig:3loopquiver_framed}), and fix the superpotential $W=A_3[A_1,A_2]$. Fix a stability parameter $\zeta^+ \in \Omega_+ = \BR_{<0}$. Fixing $n\geq 0$ and setting $\widetilde{\zeta}^+ = (-n\zeta,\zeta)$, $\widetilde n = (1,n)$, the quotient stack
\[
\mathfrak M_{\zeta^+}(\widetilde{L}_3,n) = \bigl[\Rep^{\widetilde{\zeta}^+\textrm{-st}}(\widetilde{L}_3,\widetilde n)\, \big/ \GL_n\bigr]
\]
is a smooth quasi-projective \emph{variety} of dimension $2n^2+rn$, called the \emph{noncommutative Quot scheme} in \cite{BR18}. The regular function $f_n\colon \Rep(\widetilde{L}_3,\widetilde n) \to \BA^1$ given by taking the trace of $W$ descends to a regular function on $\mathfrak M_{\zeta^+}(\widetilde{L}_3,n)$, still denoted $f_n$. We have the following description of the Quot scheme of length $n$ quotients of $\OO_{\BA^3}^{\oplus r}$.

\begin{prop}[{\cite[Thm.~2.6]{BR18}}]\label{prop:SPOT_A^3}
There is an identity of closed subschemes
\[
\Quot_{\BA^3}(\mathscr O^{\oplus r},n) = \crit(f_n) \subset \mathfrak M_{\zeta^+}(\widetilde{L}_3,n).
\]
\end{prop}  

Thanks to Proposition \ref{prop:SPOT_A^3}, we can form the virtual motives of the Quot scheme, as in \S\,\ref{subsec:virtual_motive}, and define their generating function
\begin{equation}\label{eqn:motivic_DT_rank_r}
\DT_r^{\points}(\BA^3,q) = \sum_{n\geq 0}\,\left[\Quot_{\BA^3}(\mathscr O^{\oplus r},n) \right]_{\vir}\cdot q^n\,\in\,\CM_{\BC}\llbracket q \rrbracket.
\end{equation}
\begin{remark}
The main result of \cite{BBS} is the formula
\[
\DT_1^{\points}(\BA^3,q) = \prod_{m\geq 1}\prod_{k=0}^{m-1}\left(1-\BL^{2+k-\frac{m}{2}}q^m\right)^{-1}.
\]
The series $\DT_1^{\points}(Y,q)$ studied in \cite{BBS} for an arbitrary smooth $3$-fold $Y$ also appeared in \cite{DavisonR} as the wall-crossing factor in the motivic DT/PT correspondence based at a fixed smooth curve $C \subset Y$ in $Y$.  This correspondence refined its enumerative counterpart \cite{LocalDT,Ricolfi2018}. The same phenomenon occurred in \cite{RefConifold,MorrNagao} in the context of framed motivic DT invariants. See \cite[\S\,4]{Quot19} for a generalisation $\DT_r^{\points}(Y,q)$ of \eqref{eqn:motivic_DT_rank_r} to an arbitrary smooth $3$-fold $Y$. See \cite{ricolfi2019motive} for a plethystic formula expressing the \emph{naive} motives $[\Quot_Y(F,n)] \in K_0(\Var_{\BC})$ in terms of the motives of the \emph{punctual} Quot schemes.
\end{remark}

The following consideration will constitute the final step in proving Theorem \ref{thm:main_motivic}.

\begin{lemma}\label{thm:quot_partition_function}
There is an identity
\begin{equation}\label{eqn:Partition_Function_QUOT}
\prod_{m\geq 1}\prod_{k=0}^{rm-1}\left(1-\BL^{2+k-\frac{rm}{2}}q^m\right)^{-1} \,=\, \prod_{i=1}^r \DT_1^{\points}\left(\BA^3,q\BL^{\frac{-r-1}{2}+i}\right).
\end{equation}
\end{lemma}

\begin{proof} 
The claimed identity follows from a simple manipulation:
\begin{align*}
    \prod_{i=1}^r\DT_1^{\points}(\BA^3,q\BL^{\frac{-r-1}{2}+i}) 
    &\,=\,\prod_{i=1}^r \prod_{m\geq 1}\prod_{k=0}^{m-1}\left(1-\BL^{2+k-\frac{m}{2}}\BL^{\frac{-r-1}{2}m+im}q^m\right)^{-1} \\
    &\,=\,\prod_{m\geq 1}\prod_{k=0}^{m-1}\prod_{i=1}^r \left(1-\BL^{2+k+(i-1)m-\frac{rm}{2}}q^m\right)^{-1} \\
    &\,=\,\prod_{m\geq 1}\prod_{k=0}^{rm-1}\left(1-\BL^{2+k-\frac{rm}{2}}q^m\right)^{-1}.\qedhere
\end{align*}
\end{proof}

The identification of \eqref{eqn:Partition_Function_QUOT} with $\DT_r^{\points}(\BA^3,q)$ is proven in the PhD theses of the first and third author \cite{Cazzaniga_Thesis,ThesisR}. Both proofs follow the technique introduced in the $r=1$ case by Behrend--Bryan--Szendr\H{o}i \cite{BBS}. In the next subsection, we provide a new proof of Theorem \ref{thm:main_motivic}. We exploit an $r$-framed version of motivic wall-crossing. This technique, inspired by \cite{Mozgovoy_Framed_WC,RefConifold,MorrNagao}, will be applied to small crepant resolutions of affine toric Calabi--Yau $3$-folds in \cite{Cazza_Ric}.

%%%%%%%%%%%%%%%%%%%%%%%%%%%%%%%%%
\subsection{Calculation via wall-crossing}
\label{subsec:calculation_framed_3_loop_quiver}
In this subsection we prove Theorem \ref{thm:main_motivic}. Consider the universal generating function
\[
\widetilde A_U = \sum_{n\geq 0}\,\left[\mathfrak M(\widetilde J,n) \right]_{\vir}\cdot  y^{(1,n)}\,\in\,\mathcal T_{\widetilde L_3}
\]
as an element of the motivic quantum torus. To (generic) stability parameters $\zeta^{\pm} \in \Omega_{\pm}$ we associate elements (cf.~Definition \ref{def:motivic_partition_functions})
\[
\mathsf Z_{\zeta^\pm} = \sum_{n\geq 0}\,\left[\mathfrak M_{\zeta^{\pm}}(\widetilde J,n) \right]_{\vir}\cdot  y^{(1,n)}\,\in\,\mathcal T_{\widetilde L_3}.
\]
By Lemma \ref{lemma:stability}\,\eqref{item33}, we have an identity
\begin{equation}\label{A_minus}
\mathsf Z_{\zeta^-} = y_{\infty} = y^{(1,0)},
\end{equation}
whereas the series $\mathsf Z_{\zeta^+}$ is, essentially, a ``shift'' of the generating function $\DT_r^{\points}(\BA^3,y^{(0,1)})$. More precisely, by Proposition \ref{prop:SPOT_A^3} we have an identification
\[
\mathfrak M_{\zeta^{+}}(\widetilde J,n) = \Quot_{\BA^3}(\OO^{\oplus r},n) \subset \mathfrak M_{\zeta^+}(\widetilde{L}_3,n)
\]
of critical loci sitting inside the noncommutative Quot scheme; in particular, the associated virtual motives are the same. The shift is intended in the following sense: since $\braket{(0,n),(1,0)}_{\widetilde{L}_3} = rn$, the product rule \eqref{eqn:product_in_TQ} yields the identity
\begin{equation}\label{eqn:y1n}
    y^{(1,n)} = \BL^{-\frac{rn}{2}} y^{(0,n)}\cdot y_{\infty}\,\in\,\mathcal T_{\widetilde{L}_3}.
\end{equation}
Since we can express $y^{(0,n)}$ as the $n$-fold product of $y^{(0,1)}$ with itself, we obtain 
\begin{equation}\label{eqn:A+}
\mathsf Z_{\zeta^+} = \DT_r^{\points}\left(\BA^3,\BL^{-\frac{r}{2}} y^{(0,1)}\right)\cdot y_{\infty}.
\end{equation}
The last generating function we need to analyse is
\[
A_U = \sum_{n\geq 0}\,\bigl[\mathfrak{M}(J,n) \bigr]_{\vir} \cdot y^{(0,n)}\,\in\,\mathcal T_{L_3}\,\subset\, \mathcal T_{\widetilde{L}_3},
\]
whose $n$-th coefficient is the virtual motive of the stack of $0$-dimensional $\BC[x,y,z]$-modules of length $n$. This was already computed in Example \ref{example:cut_3_loop_quiver}:
\begin{equation}\label{eqn:BU}
A_U(y) = \prod_{m \geq 1}\prod_{k \geq 1}\left(1-\BL^{2-k}y^m\right)^{-1}.
\end{equation}

The next ingredient of the proof is a particular instance of Mozgovoy's motivic wall-crossing formula \cite{Mozgovoy_Framed_WC}.

\begin{prop}\label{prop:motivic_WC_3loop_quiver}
In $\mathcal T_{\widetilde{L}_3}$, there are identities
\begin{equation}\label{eqn:WC_3loop_quiver}
\mathsf Z_{\zeta^+}\cdot A_U = \widetilde A_U = A_U \cdot \mathsf Z_{\zeta^-}.
\end{equation}
\end{prop}

\begin{proof}
Let $\widetilde{\rho} = (u,\rho)$ be a $\widetilde{J}$-module. Consider $\zeta^+\in \Omega_{+}$, and let $\braket{u,\rho}\subset \widetilde{\rho}$ be the submodule  introduced in Definition \ref{def:span}. We have that $\braket{u,\rho}$ is $\zeta^+$-stable by Lemma \ref{lemma:stability}\,\eqref{item11} and the quotient $\widetilde{\rho}/\braket{u,\rho}$ is supported at the vertex $0$. From this we deduce the decomposition $\widetilde{A}_U=\mathsf Z_{\zeta^+}\cdot A_{U}$. Consider now $\zeta^- \in \Omega_{-}$. The quotient of $\widetilde{\rho}$  by the submodule $\rho$ based at the vertex $0$ is the simple module supported at the framing vertex $\infty$. By Lemma \ref{lemma:stability}\,\eqref{item33} this is the unique $\zeta^-$-stable module for the current choice of $\zeta^-$, so we obtain the decomposition $\widetilde{A}_U=A_{U}\cdot \mathsf Z_{\zeta^-}$.
\end{proof}

\setlength{\abovecaptionskip}{-40pt}
\begin{figure}[!ht]     
     \begin{center}
     \begin{tikzpicture}[>=stealth',shorten >=1pt,auto,node distance=3cm,
  thick,main node/.style={circle,draw,minimum size=3mm}]
  
  \node at (-3,-1) {$\Omega_{+}$};
  \node at (0,-1) {$\Omega_{0}$};
  \node at (3,-1) {$\Omega_{-}$};
  \node (a) at (0,3) {};
  \node (b) at (0,-3) {};
  \draw (-5,0)--(5,0);
  \draw[line width=1.8] (0.2,-0.2)--(-0.2,0.2);
  \draw[line width=1.8] (-0.2,-0.2)--(0.2,0.2);
 
   \begin{pgflowlevelscope}{\pgftransformscale{0.45}} 
        \node[main node](0) at (8.5-16,6-2)  {$v_{0}$};
        \node[circle,draw,minimum size=3mm] (i) at (8.5-16,4-2)  {$v_{\infty}$};
        \node at (7.5-16,5.5-2) {$\subseteq$};
        \node (dots) at (8.5-16,5-2) {$\ldots$};
        \node[main node](0') at (6.5-16,6-2)  {$v_{0}$};
        \node[circle,draw,minimum size=3mm] (i') at (6.5-16,4-2)  {$v_{\infty}$};
        \node (dots) at (6.5-16,5.5-2) {$\ldots$};
        \node (n) at (9.2-16,6-2) {$n$};
        \node (n') at (5.7-16,6-2) {$n'$};
        \node (1) at (9.2-16,4-2) {$1$};
        \node (1') at (5.7-16,4-2) {$1$};
        \path[every node/.style={font=\sffamily\small,
  	    	fill=white,inner sep=1pt},every loop/.style={looseness=12}]
  	        % Right-hand-side arrows rendered from top to bottom to
  	        % achieve proper rendering of labels over arrows.
            (0) edge [->,in=110, out=70,  loop] node {}(0)
            (0) edge [->,in=120, out=60, loop] node {}(0)    
            (0) edge [->,in=130, out=50, loop] node {}(0)
            (i) 
                edge [->,bend left=20] node[] {} (0)        
                edge [->,bend left=-20] node[] {} (0); 
        
        \path[every node/.style={font=\sffamily\small,
  		    fill=white,inner sep=1pt},every loop/.style={looseness=12}]
  	        % Right-hand-side arrows rendered from top to bottom to
  	        % achieve proper rendering of labels over arrows.
            (0') edge [->,in=110, out=70,  loop] node {}(0')
            (0') edge [->,in=120, out=60, loop] node {}(0')    
            (0') edge [->,in=130, out=50, loop] node {}(0')
            (i') 
                edge [->,bend left=20] node[] {} (0')        
                edge [->,bend left=-20] node[] {} (0'); 
    \end{pgflowlevelscope}  
  %% OMEGA_- chamber    
  \begin{pgflowlevelscope}{\pgftransformscale{0.45}} 
        \node[main node](0) at (-8.5+16,-4.5+8)  {$v_{0}$};
        \node at (-7.5+16,-4.5+8) {$\subseteq$};
        \node (dots) at (-6.5+16,-5+8) {$\ldots$};
        \node[main node](0') at (-6.5+16,-4+8)  {$v_{0}$};
        \node[circle,draw,minimum size=3mm] (i') at (-6.5+16,-6+8)  {$v_{\infty}$};
        \node (n) at (-9.2+16,-4.5+8) {$n$};
        \node (n') at (-5.7+16,-4+8) {$n$};
        \node (1') at (-5.7+16,-6+8) {$1$};
        % Right-hand-side arrows rendered from top to bottom to
  	    % achieve proper rendering of labels over arrows.
            \path[every node/.style={font=\sffamily\small,
  	    	fill=white,inner sep=1pt},every loop/.style={looseness=12}]
  	        % Right-hand-side arrows rendered from top to bottom to
  	        % achieve proper rendering of labels over arrows.
            (0) edge [->,in=110, out=70,  loop] node {}(0)
            (0) edge [->,in=120, out=60, loop] node {}(0)    
            (0) edge [->,in=130, out=50, loop] node {}(0);
            \path[every node/.style={font=\sffamily\small,
  		    fill=white,inner sep=1pt},every loop/.style={looseness=12}]
  	        % Right-hand-side arrows rendered from top to bottom to
  	        % achieve proper rendering of labels over arrows.
            (0') edge [->,in=110, out=70,  loop] node {}(0')
            (0') edge [->,in=120, out=60, loop] node {}(0')    
            (0') edge [->,in=130, out=50, loop] node {}(0')
            (i') 
                    edge [->,bend left=20] node[] {} (0')        
                    edge [->,bend left=-20] node[] {} (0');  		
  \end{pgflowlevelscope}    
  %CRITICAL CHAMBER    
  \begin{pgflowlevelscope}{\pgftransformscale{0.45}} 
        \node[main node](0) at (.5,4)  {$v_{0}$};
        \node[circle,draw,minimum size=3mm] (i) at (.5,2)  {$v_{\infty}$};
        \node (dots) at (.5,3) {$\ldots$};
        \node (dots) at (-.5,3.5) {$\subseteq$};
        \node (dots) at (-1.3,3.5) {$0$};
        \node (n) at (-.2,4) {$n$};
        \node (1) at (-.2,2) {$1$};
  
        \path[every node/.style={font=\sffamily\small,
  		        fill=white,inner sep=1pt},every loop/.style={looseness=12}]
  	            % Right-hand-side arrows rendered from top to bottom to
  	            % achieve proper rendering of labels over arrows.
                (0) edge [->,in=110, out=70,  loop] node {}(0)
                (0) edge [->,in=120, out=60, loop] node {}(0)    
                (0) edge [->,in=130, out=50, loop] node {}(0)
                (i) 
                    edge [->,bend left=20] node[] {} (0)        
                    edge [->,bend left=-20] node[] {} (0); 
  \end{pgflowlevelscope}
\end{tikzpicture}
\caption{An illustration of Proposition \ref{prop:motivic_WC_3loop_quiver}.}\label{fig_wall}
\end{center}
\end{figure}
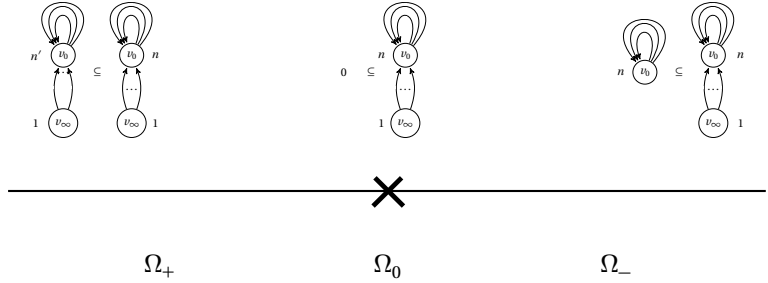

Note that we have an identity
\[
y_{\infty}\cdot y^{(0,n)} = \BL^{-\frac{rn}{2}}\cdot y^{(1,n)} = \BL^{-rn}\cdot y^{(0,n)}\cdot y_{\infty},
\]
where we have used \eqref{eqn:y1n} for the second equality.
By Formula \eqref{eqn:A+}, the left-hand term of Formula \eqref{eqn:WC_3loop_quiver} can then be rewritten as
\begin{multline*}
\DT_r^{\points}\left(\BA^3,\BL^{-\frac{r}{2}} y^{(0,1)}\right)\cdot \sum_{n\geq 0}\,\bigl[\mathfrak{M}(J,n)\bigr]_{\vir}\cdot y_{\infty} \cdot y^{(0,n)} \\
=\DT_r^{\points}\left(\BA^3,\BL^{-\frac{r}{2}} y^{(0,1)}\right)\cdot \sum_{n\geq 0}\,\bigl[\mathfrak{M}(J,n)\bigr]_{\vir} \BL^{-rn}\cdot y^{(0,n)}\cdot y_{\infty}.
\end{multline*}

Therefore, by Equation \eqref{A_minus}, identifying the two expressions for $\widetilde{A}_U$ in Equation \eqref{eqn:WC_3loop_quiver} yields
\[
\DT_r^{\points}\left(\BA^3,\BL^{-\frac{r}{2}} y^{(0,1)}\right)\cdot A_U\left(\BL^{-r}y^{(0,1)}\right) \cdot y_{\infty} = A_U\left(y^{(0,1)}\right)\cdot y_{\infty},
\]
which is equivalent to
\[
\DT_r^{\points}\left(\BA^3,\BL^{-\frac{r}{2}} y^{(0,1)}\right) = \frac{A_U\left(y^{(0,1)}\right)}{A_U\left(\BL^{-r}y^{(0,1)}\right)}.
\]
Setting $q = \BL^{-\frac{r}{2}}y^{(0,1)}$, and using Equation \eqref{eqn:BU}, a simple substitution yields
\begin{align*}
    \DT_r^{\points}(\BA^3,q) &= \frac{A_U\left(\BL^{\frac{r}{2}}q\right)}{A_U\left(\BL^{-\frac{r}{2}}q\right)} \\
    &=\prod_{m\geq 1}\prod_{j\geq 0}\frac{\left(1-\mathbb L^{1-j+\frac{rm}{2}}q^m\right)^{-1}}{\left(1-\mathbb L^{1-j-\frac{rm}{2}}q^m\right)^{-1}}\\
&=\prod_{m\geq 1}\prod_{j=0}^{rm-1}\left(1-\mathbb L^{1-j+\frac{rm}{2}}q^m\right)^{-1}\\
&=\prod_{m\geq 1}\prod_{k=0}^{rm-1}\left(1-\mathbb L^{2+k-\frac{rm}{2}}q^m\right)^{-1}.
\end{align*}
Formula \eqref{formu} is proved. By Lemma \ref{thm:quot_partition_function}, Theorem \ref{thm:main_motivic} is proved.

\begin{remark}
The generating function $\DT_r^{\points}(\BA^3,(-1)^rq)$ admits the plethystic expression
\[
\DT_r^{\points}(\BA^3,(-1)^rq) = 
\Exp\left(\frac{(-1)^rq\mathbb L^{\frac{3}{2}}}{\bigl(1-(-\mathbb L^{-\frac{1}{2}})^rq\bigr)\bigl(1-(-\mathbb L^{\frac{1}{2}})^rq\bigr)}\frac{\mathbb L^{-\frac{r}{2}}-\mathbb L^{\frac{r}{2}}}{\mathbb L^{-\frac{1}{2}}-\mathbb L^{\frac{1}{2}}}\right).
\]
This is exploited in \cite[\S\,4]{Quot19} to define $[\Quot_Y(F,n)]_{\vir}$ for all $3$-folds $Y$ and locally free sheaves $F$ over $Y$.
\end{remark}

\begin{remark}
The Euler number specialisation $\BL^{1/2} \to -1$ applied to Formula \eqref{formu} yields a formula for the generating function of \emph{virtual Euler characteristics} of Quot schemes,
\[
\sum_{n \geq 0} \chi_{\vir}(\Quot_{\BA^3}(\OO^{\oplus r},n))\cdot q^n = \mathsf M((-1)^rq)^r,
\]
where $\mathsf M(q) = \prod_{m\geq 1}(1-q^m)^{-m}$ is the MacMahon function, the generating function of plane partitions, and where $\chi_{\vir}(-) = \chi(-,\nu)$ is the Euler characteristic weighted by Behrend's microlocal function \cite{Beh}. The above identity can be seen as the Calabi--Yau specialisation (i.e.~the specialisation $s_1+s_2+s_3 = 0$) of the generating function of \emph{cohomological Donaldson--Thomas invariants} of $\BA^3$,
\[
\sum_{n\geq 0}\DT_r^{\coh}(\BA^3,n)\cdot q^n = \mathsf M((-1)^rq)^{-r\frac{(s_1+s_2)(s_1+s_3)(s_2+s_3)}{s_1s_2s_3}}\,\,\in\,\, \BQ(\!(s_1,s_2,s_3)\!) \llbracket q \rrbracket,
\]
obtained in \cite[Thm.~B]{FMR_K-DT} (as a higher rank version of \cite[Thm.~1]{MNOP2}), where $s_1$, $s_2$ and $s_3$ are the equivariant parameters of the torus $\BT = \BG_m^3$ acting on the Quot scheme.
\end{remark}

\section{The normal limit law and asymptotics}\label{sec:Asymptotics}

In \S\,\ref{sec:random_variables_on_colored_partitions}, we introduce a family of random variables on the space of $r$-colored plane partitions, and we describe the asymptotics of the members of the family after suitable normalisation in Proposition \ref{propCLT}. Theorem \ref{main1}, the main theorem of the section, is deduced from Proposition \ref{propCLT} in subsection \ref{sub:proofB}. Finally, subsection \ref{sec:proofCLT} is entirely devoted to the proof of Proposition \ref{propCLT}.    

\subsection{Random variables on $r$-colored plane partitions}\label{sec:random_variables_on_colored_partitions}

We introduce a multivariate function
\begin{equation}\label{eq:F_prod}
    F(u,v,w,z)=\prod_{l=1}^{r}\prod_{m\geq 1}\prod_{k=1}^{m}\left(\
1-wu^{k}v^{ml}z^m\right)^{-1}.
\end{equation}
The coefficient of $z^n$ is a polynomial on the three variables $u$, $v$ and $w$, which we denote by $Q_{n}(u,v,w)$, whose coefficients are nonnegative integers. When $u=v=w=1$, we obtain the well known MacMahon function raised to the power $r$, which is the generating function for  the $r$-tuples of plane partitions. Hence, $Q_{n}(1,1,1)$ is the number of $r$-tuples of plane partitions of total size $n$, i.e.~the number of $r$-colored plane partitions $(\pi_1,\, \pi_2, \, \dots, \, \pi_r)$ such that $\sum_{j=1}^{r}\lvert\pi_j\rvert=n$, where $\lvert\pi_j\rvert$ denotes the sum of the entries of the plane partition $\pi_j$ (or the number of boxes, cf.~Figure \ref{partition}). The polynomial $Q_{n}(u,v,w)$, when divided by $Q_{n}(1,1,1)$, represents the joint probability generating function of some random variables $X_n$, $Y_n$, and $Z_n$ on the space of $r$-colored plane partitions of size $n$, where each $r$-tuple is equally likely. More precisely, we have
\begin{equation}\label{eq:prob_gen}
\frac{Q_{n}(u,v,w)}{Q_{n}(1,1,1)}=\mathbb{E}\left(u^{X_n}v^{Y_n}w^{Z_n}\right).
\end{equation} 
To describe these random variables, we need to define certain parameters of plane partitions. For a plane partition $\pi$, let $\Delta(\pi)$ denote the sum of the diagonal parts of $\pi$,  $\Delta_{+}(\pi)$ denote the sum of the upper diagonal parts,  and $\Delta_{-}(\pi)$ denote the sum of the lower diagonal parts. See Figure~\ref{partition} for an example of a plane partition showing the values of these parameters.

\setlength{\abovecaptionskip}{10pt}
\begin{figure}[!ht]
\begin{tikzpicture}[scale=0.36] 
\centering
\node[anchor=west] at (-15,1) 
  {\begin{ytableau}
   1 \\
   2 & 2 \\
   2 & 2 & 1 \\
   3 & 3 & 2 \\
   5 & 4 & 3 & 1  \\
  \end{ytableau}};
\planepartition{{5,3,2,2,1},{4,3,2,2},{3,2,1},{1}}
\end{tikzpicture}
\caption{A plane partition $\pi$ of size $\lvert\pi\rvert=31$, $\Delta(\pi)=9$, $\Delta_{+}(\pi)=12$, and $\Delta_{-}(\pi)=10$.}\label{partition}
\end{figure}
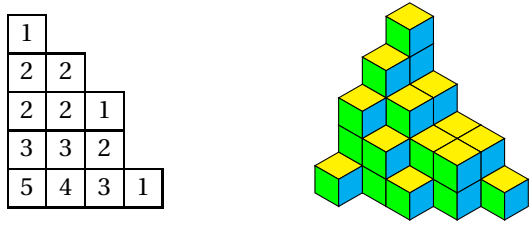

The parameter $\Delta(\pi)$ is also known as the trace of $\pi$ and it has been studied in the literature, see for instance \cite{Kamenov2007} and the references therein. In particular, one has 
\[
\prod_{m\geq 1}(1-wz^m)^{-m}=\sum_{\pi}w^{\Delta(\pi)}z^{\lvert\pi\rvert}.
\]
Similarly, we can find in \cite{Morrison_asymptotics} that 
\[
\prod_{m\geq 1}\prod_{k=1}^{m}(1-q^{2k-m}z^{m})^{-1}=\sum_{\pi}q^{\Delta(\pi)+\Delta_{+}(\pi)-\Delta_{-}(\pi)}z^{\lvert\pi\rvert}.
\]
We can easily deduce from these two identities that for an $r$-colored plane partition $\overline{\pi} = (\pi_1, \pi_2, \ldots, \pi_r)$ of total size $n$, we have
\begin{align*}
    X_n\left(\overline{\pi}\right)& =\frac{n}{2}+\frac{1}{2}\left( \sum_{l=1}^{r}\left(\Delta(\pi_l)+\Delta_{+}(\pi_l)-\Delta_{-}(\pi_l)\right)\right)\\
    &=\sum_{l=1}^{r}\left(\Delta(\pi_l)+\Delta_{+}(\pi_l)\right),\\
    Y_n\left(\overline{\pi}\right)& = \sum_{l=1}^{r}l|\pi_l|, \ \ \text{and}\\
    Z_n\left(\overline{\pi}\right)& = \sum_{l=1}^{r}\Delta(\pi_l).
\end{align*}
When $r=1$, Kamenov and Mutafchiev~\cite{Kamenov2007} proved that the distribution of the trace of a random plane partition of size $n$, when suitably normalised, is asymptotically normal. Morrison  \cite{Morrison_asymptotics} also established asymptotic normality for any random variable of the form $\delta\Delta(\pi)+\Delta_{+}(\pi)-\Delta_{-}(\pi)$, where $\pi$ is a random plane partition of size $n$ and $\delta$ is a fixed real number. We show that for any fixed integer $r\geq 1$, any nontrivial linear combination of the variables $X_n$, $Y_n$ and $Z_n$, when suitably normalised, converges weakly to a normal distribution. It is worth noting that the random variable $Y_n$ is non-constant only when $r>1$.

\bprop \label{propCLT}
For any fixed real vector $(\alpha,\beta,\gamma)\neq (0,0,0)$, there exist sequences of real numbers $\mu_n$ and $\sigma_n\geq 0$ such that the normalised random variable 
$$\frac{\alpha X_n+\beta Y_n+ \gamma Z_n-\mu_n}{\sigma_n}$$
converges weakly to the standard normal distribution. Moreover, $\mu_n$ and $\sigma_n$ satisfy the following asymptotic formulas as $n\to \infty$
\begin{align*}
\mu_n&\,\,=\,\,\left(\frac{1}{2}\alpha+\frac{r+1}{2} \beta\right)n+\frac{r^{1/3}\zeta(2)(\alpha+2\gamma)}{2^{5/3}(\zeta(3))^{2/3}}\, n^{2/3}+\mathcal{O}(n^{1/3}), \, \text{ and }\,\\
\sigma_n^2 &\,\,\sim\,\, 
\begin{cases}
\displaystyle \frac{\alpha^2+(r^2-1)\beta^2}{2^{7/3}(r\zeta(3))^{1/3}}\, n^{4/3}\, & \text{ if } (\alpha,\beta)\neq (0,0),\\[2em]
\displaystyle \hfil \frac{r^{1/3}\gamma^2}{3(2\zeta(3))^{2/3}} n^{2/3}\log n \, & \text{otherwise.} 
\end{cases}
\end{align*}
\eprop
Looking at the asymptotic behaviour of the random variables $X_n$, $Y_n$ and $Z_n$, when divided by $n^{2/3}$, we observe from the above result that the random variable $n^{-2/3}Z_n$ degenerates as $n\to \infty$. Furthermore, by the Cam\'er-Wold device \cite{Cramer-Wold}, the random variables $n^{-2/3}X_n$, $n^{-2/3}Y_n$ converge jointly to a bivariate normal distribution with a diagonal covariance matrix. The appropriate normalisation of $Z_n$ is $n^{-1/3}(\log n)^{-1/2}Z_n$ which is, when centred, asymptotically normal. This asymptotic normality and the asymptotic formulas for $\mu_n$ and $\sigma_n^2$ agree with the main result in \cite{Kamenov2007} when $r=1$ and $(\alpha,\beta,\gamma)=(0,0,1).$

\begin{convention}
We shall use the Vinogradov notation $\ll$ interchangeably with the $\mathcal O$-notation. For instance, by $f(n)\ll g(n)$ (or $g(n)\gg f(n)$)  as $n\to\infty$, we mean that there exists a positive constant $C$ such that $|f(n)|\leq Cg(n)$ for sufficiently large $n$. 
\end{convention}

Theorem \ref{main1} now follows immediately from Proposition \ref{propCLT} as we will see next.

\subsection{Deducing Theorem \ref{main1}}\label{sub:proofB} 
Granting Proposition \ref{propCLT}, we can now finish the proof of our second main result.

\begin{proofof}{Theorem \ref{main1}}
With the change of variable $T=\BL^{1/2}$, the combination of the equations \eqref{formu} and \eqref{formula_product} in Theorem~\ref{thm:main_motivic} yields 
\[
\DT_r^{\points}(\BA^3,q)
 =\prod_{l=1}^{r}\prod_{m\geq 1}\prod_{k=0}^{m-1}\left(1-T^{4+2k-m}\Big(qT^{-r-1+2l}\Big)^m\right)^{-1}.
\]
The product on the right-hand side can be expressed in terms of our auxiliary function $F(u,v,w,z)$ defined at the beginning of this section. First we write it as follows
\[
     \prod_{l=1}^{r}\prod_{m\geq 1}\prod_{k=0}^{m-1}\left(1-T^{4+m-2(m-k)}\Big(qT^{r+1-2(r-l+1)}\Big)^m\right)^{-1}.
\]
If $k$ goes from $0$ to $m-1$, then $m-k$ goes from $m$ to $1$. Similarly, if $l$ goes from $1$ to $r$, then $r-l+1$ goes from $r$ to $1$. Therefore, we have
\begin{align*}
\DT_r^{\points}(\BA^3,q)
& =\prod_{l=1}^{r}\prod_{m\geq 1}\prod_{k=1}^{m}\left(1-T^{4+m-2k}(T^{r+1-2l}q)^m\right)^{-1}\\
& = \prod_{l=1}^{r}\prod_{m\geq 1}\prod_{k=1}^{m}\left(1-T^{4-2k-2ml}(T^{r+2}q)^m\right)^{-1}\\
& =F(T^{-2},T^{-2},T^4,T^{r+2}q).
\end{align*}
This implies that 
\begin{equation}\label{eq:ET}
\BE(T^{S_{n,r}})= \frac{M_{n,r}(T)}{M_{n,r}(1)}=\BE\left(T^{4Z_n-2X_n-2Y_n+(r+2)n}\right),
\end{equation}
where $M_{n,r}(T)$ is, as defined in \eqref{eq:coefM}, the coefficient of $q^n$ in $\DT_r^{\points}(\BA^3,q)$. 
The second equality in \eqref{eq:ET} makes use of Equation~\eqref{eq:prob_gen}. Thus, $S_{n,r}$ has the same distribution as  $4Z_n-2X_n-2Y_n+(r+2)n$ --- a shifted linear combination of the variables $X_n$, $Y_n$, and $Z_n$. Now, applying Proposition \ref{propCLT} with $(\alpha, \beta, \gamma) = (-2, -2, 4)$, we deduce that the normalised random variable  $n^{-2/3}(4Z_n-2X_n-2Y_n+(r+2)n)$ converges weakly to the normal distribution $\mathcal{N}(\mu,\sigma^2)$ with 
\[
\mu = \frac{r^{1/3}\pi^2}{2^{5/3}(\zeta(3))^{2/3}}\, \text{ and }\, \sigma^2=\frac{r^{5/3}}{(2\zeta(3))^{1/3}},
\] 
which proves Theorem \ref{main1}. 
\end{proofof}

\subsection{Proof of Proposition \ref{propCLT}}\label{sec:proofCLT}
Morrison used the method of moments to prove his result in \cite{Morrison_asymptotics}. However, due to the appearance of the second variable $Y_n$ and the complication that comes with it, we decided to use a different approach. We follow the method that Hwang used in \cite{Hwang} to prove limit theorems for the number of parts in the so-called restricted partitions (these are one dimensional partitions with some restrictions on the parts).  The first part of the proof is based on the saddle-point method to get an asymptotic formula for   $Q_{n}(u,v,w)$ as $n\to\infty$, and the second is a perturbation technique to deduce the central limit theorem. 
\subsubsection{Saddle-point method}
The goal here is to obtain an asymptotic formula for $Q_{n}(u,v,w)$ as $n\to\infty$, where $(u,v,w)$ is allowed to vary in a fixed real neighborhood of $(1,1,1)$. To simplify our notation,  define
$
\Phi(t)=-\log(1-e^{-t}) 
$,
and for real numbers $a$, $b$ and $c$, we let 
\[
f(x,y)=\sum_{l=1}^{r}\sum_{m=1}^{\infty}
\sum_{k=1}^{m}\Phi(xm+y(c+ak+mbl)).
\]
The function $f$ depends on $(a,b,c)$ but we drop this dependence for now to ease notation. Also, for a positive number $\rho$, we make the  substitution $z=e^{-\tau}$, where $\tau=\rho+it$. Hence, 
\[
f(\tau,\rho)
=
-\sum_{l=1}^{r}\sum_{m=1}^{\infty}\sum_{k=1}^{m}\log\left(1-e^{-\rho(c+ak+mbl)-\tau m}\right)=\log F(e^{-a\rho},e^{-b\rho},e^{-c\rho},e^{-\tau}).
\]
One can easily verify that if $(u,v,w)$ is bounded (which is the case throughout this section), then there exists a fixed positive real number $R$ such that  the product in \eqref{eq:F_prod} converges absolutely whenever $|z|<R$. Hence, $F(u,v,w,z)$, as function of $z$, is analytic in a complex neighborhood of $0$. By Cauchy's integral formula, we have
\begin{equation}\label{eq:Cauchy}
Q_{n}(e^{-a\rho},e^{-b\rho},e^{-c\rho})=\frac{e^{n\rho}}{2\pi}\int_{-\pi}^{\pi}\exp\Big(f(\rho+it,\rho)+nt i\Big)dt.
\end{equation}
We now use the saddle-point method to estimate the above integral. We choose $\rho$ to be the positive solution of the equation
\begin{equation}\label{eq:saddle}
n=-f_{x}(\rho,\rho)=\sum_{l=1}^{r}\sum_{m=1}^{\infty}\sum_{k=1}^{m}\frac{me^{-\rho(m+(c+ak+mbl))}}{1-e^{-\rho(m+(c+ak+mbl))}},
\end{equation}
where $f_x$ denotes the partial derivative of $f$ with respect to $x$. Similar notations will be used for other partial derivatives. Note that there is a unique positive solution $\rho=\rho(n,a,b,c)$ of Equation~\eqref{eq:saddle} since the function defined by the series is strictly decreasing as a function of $\rho$, provided that $a$, $b$  and $c$ are small enough (it suffices for instance to assume that $|c|+|a|+r|b|<1$). Furthermore, we observe that $\rho\to0$ as $n\to \infty$. The following lemma reveals the asymptotic dependence between $n$ and $\rho$. 
\begin{lemma}\label{lem:1}
Let $\epsilon$ be a number in the interval  $[0,1/2]$ and $\rho$ be the solution of Equation~\eqref{eq:saddle}. Then we have 
\begin{equation}
\frac{2r\zeta(3)}{(1+\epsilon)^3} \, \rho^{-3}+\mathcal{O}(\rho^{-2})\leq n \leq \frac{2r\zeta(3)}{(1-\epsilon)^3}\,  \rho^{-3}+\mathcal{O}(\rho^{-2}),
\end{equation}
as $n\to\infty$, uniformly for $|c|+|a|+r|b|\leq \epsilon$, where the implied constants in  the $\mathcal{O}$-terms are independent of $\epsilon.$
\end{lemma}
\begin{proof}
Recall from \eqref{eq:saddle} that
\[
n =-f_x(\rho,\rho)=-\sum_{l=1}^{r}\sum_{m=1}^{\infty}\sum_{k=1}^{m}
m\Phi'\left(\rho m\Big(1+\frac{c+ak+mbl}{m}\Big)\right).
\]
Under the  assumption that $|c|+|a|+r|b|\leq \epsilon$, for any $m\geq 1$, $1\leq k\leq m$ and $1\leq l\leq r$, we have 
\[
|c+ak+mbl|\leq \left(|c|+|a|+r|b|\right) m\leq \epsilon m.
\]
Moreover, the function $\Phi'(x)$ is an increasing function. Therefore,  
\[
\Phi'((1-\epsilon)\rho m)\leq \Phi'\left(\rho m\left(1+\frac{c+ak+mbl}{m}\right)\right)\leq \Phi'( (1+\epsilon)\rho m).
\] 
Multiplying by $m$ and summing over  $m\geq 1$, $1\leq k\leq m$ and $1\leq l\leq r$, we obtain 
\[
f_x((1-\epsilon)\rho,0)\leq f_x(\rho,\rho)\leq f_x((1+\epsilon)\rho,0).
\]
We can obtain asymptotic estimates of the lower and upper bounds as $\rho\to0^+$. This can be done via Mellin transform. The reader can consult \cite{Mellin} for a comprehensive survey on the Mellin transform method. The Mellin transform of $f_x((1-\epsilon)t,0)$ is  
\[
\int_{0}^{\infty} f_x((1-\epsilon)t,0) t^{s-1} dt = -r(1-\epsilon)^{-s}\zeta(s-2)\zeta(s)\Gamma(s),
\]
which has simple poles at $s=3$ and $s=1$. The other singularities are precisely at the negative odd integers. Thus, we have
\begin{equation}\label{eq:mellin}
    f_x((1-\epsilon)\rho,0)=-\frac{2r\zeta(3)}{(1-\epsilon)^{3}}\, \rho^{-3}+\frac{r}{12(1-\epsilon)}\, \rho^{-1}-\frac{r}{2\pi i }\int_{-i\infty}^{i\infty}\zeta(s-2)\zeta(s)\Gamma(s)((1-\epsilon)\rho)^{-s}ds.
\end{equation}
Since $|\zeta(it-2)\zeta(it)\Gamma(it)|$ decays exponentially fast as $t\to \pm\infty$, the absolute value of the  integral on the right hand side is bounded by an absolute constant. The same argument works for the estimate of the upper bound $f_x((1+\epsilon)\rho,0)$. This completes the proof of the lemma.
\end{proof}
Next, we split the integral on the right-hand side of \eqref{eq:Cauchy} into two parts as follows: let 
\[
\mathcal{I}_1=\frac{e^{n\rho}}{2\pi}\int_{-\rho^{C}}^{\rho^{C}}\exp\Big(f(\rho+it,\rho)+nti\Big)\,dt,
\] 
where $C$ is an absolute constant in the interval $(5/3, 2)$, and let $\mathcal{I}_2=Q_{n}(e^{-a\rho},e^{-b\rho},e^{-c\rho})-\mathcal{I}_1$.

\subsection*{Estimate of $\mathcal{I}_1$} 
For the rest of this section, we work under the condition of Lemma \ref{lem:1}, that is $|c|+|a|+r|b|\leq \epsilon$ and $\epsilon\in [0,1/2]$. For $-\rho^{C}\leq t\leq \rho^{C}$, Equation~\eqref{eq:saddle} and a Taylor approximation of $f(\rho+it,\rho)$ give  
\[
f(\rho+it,\rho)+nit=f(\rho,\rho)-f_{xx}(\rho,\rho)\frac{t^2}{2}+\mathcal{O}\left(\rho^{3C}\max_{-\rho^{\rho}\leq \theta\leq \rho^{C} }\Big|f_{xxx}(\rho+i\theta,\rho)\Big|\right),
\]
where the implied constant in the error term is absolute. To estimate the error term, observe that 
\[
f_{xxx}(\tau,\rho)=-\sum_{l=1}^{r}\sum_{m=1}^{\infty}\sum_{k=1}^{m}\frac{m^3e^{-\tau m-\rho(c+ak+mbl)}(1+e^{-\tau m-\rho(c+ak+mbl)})}{(1-e^{-\tau m-\rho(c+ak+mbl)})^3}
\]
For any real number $\theta$ and $\tau=\rho+i\theta$ we have
\begin{align*}
|1+e^{-\tau m-\rho(c+ak+mbl)}|& \leq 1+e^{-\rho m-\rho(c+ak+mbl)}, \\
|1-e^{-\tau m-\rho(c+ak+mbl)}| & \geq 1-e^{-\rho m-\rho(c+ak+mbl)}.
\end{align*}
Hence 
$
|f_{xxx}(\rho+i\theta,\rho)|
$
is bounded above by $|f_{xxx}(\rho,\rho)|$. We can estimate $|f_{xxx}(\rho,\rho)|$ as we did for $f_{x}(\rho,\rho)$ in the proof of Lemma \ref{lem:1}. We obtain
$$
|f_{xxx}(\rho+i\theta,\rho)|\leq |f_{xxx}(\rho,\rho)| =\mathcal{O}(\rho^{-5}),
$$  
where the implied constant depends only on $r$. Therefore, for $\lvert t \rvert\leq \rho^{C}$ we have
\[
f(\rho+it,\rho)+nit=f(\rho,\rho)-f_{xx}(\rho,\rho)\frac{t^2}{2}+\mathcal{O}(\rho^{3C-5}).
\]
Since we chose $C>5/3$, we have $\rho^{3C-5}=o(1)$. Thus  
\[
\mathcal{I}_1=\frac{e^{f(\rho,\rho)+n\rho}}{2\pi}\int_{-\rho^C}^{\rho^C}e^{-f_{xx}(\rho,\rho)t^2/2}\, dt\, \Big(1+\mathcal{O}(\rho^{3C-5})\Big).
\]
It remains to estimate the integral on the right hand side as $\rho\to 0^+$. Note that  $f_{xx}(\rho,\rho)>0$ and $f_{xx}(\rho,\rho) \gg \rho^{-4} $ (again via Mellin transform as in Lemma~\ref{lem:1}), so we have 
\begin{align*}
\int_{-\rho^{C}}^{\rho^{C}}e^{-f_{xx}(\rho,\rho)t^2/2}\, dt
& = \int_{-\infty}^{\infty}e^{-f_{xx}(\rho,\rho)t^2/2}\, dt-2\int_{\rho^{C}}^{\infty}e^{-f_{xx}(\rho,\rho)t^2/2}\, dt\\
& =\sqrt{\frac{2\pi}{f_{xx}(\rho,\rho)}}+\mathcal{O}\left(\int_{\rho^{C}}^{\infty}e^{-\rho^{C}f_{xx}(\rho,\rho)t/2}\, dt\right)\\
& =\sqrt{\frac{2\pi}{f_{xx}(\rho,\rho)}}+\mathcal{O}\left(\rho^{4-c}e^{-A\rho^{2C-4}}\right),
\end{align*}
where $A>0$ and the hidden constants in the error terms above depend only on $r$. Thus, since we chose $C<2$, the term $\rho^{4-c}e^{-A\rho^{2C-4}}$ tends to zero faster than any power of $\rho$ as $\rho\to0^{+}$. Hence, we obtain an estimate for $\mathcal{I}_1$
\begin{equation}\label{eq:est_I1}
\mathcal{I}_1=\frac{e^{f(\rho,\rho)+n\rho}}{\sqrt{2\pi f_{xx}(\rho,\rho)}}\Big(1+\mathcal{O}(\rho^{3C-5})\Big)\ \ \text{as}\ \ \rho\to 0^{+}.
\end{equation}
This estimate holds uniformly for $|c|+|a|+r|b|\leq \epsilon$ and $\epsilon\in [0,1/2]$.
\subsection*{Estimate of $\mathcal{I}_2$}
We will prove that  $|\mathcal I_2|$ is much smaller than $|\mathcal I_1|$. To this end, we assume that $\rho^{C}< t \leq \pi$, where $C$ is as before.  We have 
\begin{align*}
 \mathrm{Re} (f(\rho+it,\rho))-f(\rho,\rho)
& \,\,=\,\,-\sum_{l=1}^r\sum_{m=1}^{\infty}\sum_{k=1}^{m}
\sum_{j=1}^{\infty}j^{-1}e^{-j\rho(m+(c+ak+mbl))}(1-\cos(mjt))\\
& \,\,\leq\,\, -\sum_{l=1}^r\sum_{m=1}^{\infty}\sum_{k=1}^{m}
e^{-\rho(m+(c+ak+mbl))}\left(1-\cos(mt)\right)\\
& \,\,\leq\,\, -r\sum_{m=1}^{\infty}me^{-\rho(1+\epsilon)m}(1-\cos(mt)). 
\end{align*}
Moreover, 
\[
\sum_{m=1}^{\infty}me^{-\rho(1+\epsilon)m}(1-\cos(mt))
 =\frac{e^{\rho(1+\epsilon)}}{(e^{\rho(1+\epsilon)}-1)^2}-\mathrm{Re}\left( \frac{e^{\rho(1+\epsilon)+it}}{(e^{\rho(1+\epsilon)+it}-1)^2}\right).
\]
A lower estimate of the same term can be found in the proof of \cite[Lemma~5]{root}. By the same argument as the one given in loc.~cit., but with $\lvert t \rvert\geq \rho^{C}$, we get  
\[
\sum_{m=1}^{\infty}me^{-\rho(1+\epsilon)m}(1-\cos(mt)) \gg  \left(\rho(1+\epsilon)\right)^{2C-4} \ \ \text{as}\ \ \rho\to0^+,
\]
where the implied constant is independent of $\epsilon.$ 

Noting that $2C-4<0$, we deduce that $\exp\left( \mathrm{Re} (f(\rho+it,\rho))-f(\rho,\rho)\right)$ tends to zero faster than any power of $\rho$ as $\rho\to0^+$. Thus, by \eqref{eq:est_I1}, we find
\[
\frac{|\mathcal{I}_2|}{|\mathcal{I}_1|}\ll \sqrt{f_{xx}(\rho,\rho)} \int_{\rho^C}^{\pi}\exp\left( \mathrm{Re} (f(\rho+it,\rho))-f(\rho,\rho)\right)dt,
\]
which tends to zero faster than any power of $\rho$ as $\rho\to0^+$. Recalling that $Q_{n}(e^{-a\rho},e^{-b\rho},e^{-c\rho})=\mathcal{I}_1+\mathcal{I}_2$, we finally obtain
\begin{equation}\label{eq:asym_Q_n}
    Q_{n}(e^{-a\rho},e^{-b\rho},e^{-c\rho})=\frac{e^{f(\rho,\rho)+n\rho}}{\sqrt{2\pi f_{xx}(\rho,\rho)}}\Big(1+\mathcal{O}(\rho^{3C-5})\Big)\ \ \text{as}\ \ n\to \infty,
\end{equation}
uniformly for $|c|+|a|+r|b|\leq \epsilon$ and $\epsilon\in [0,1/2]$, where $\rho$ is the solution of Equation~\eqref{eq:saddle}.
\subsubsection{Perturbation}
Here we set $u=e^{\eta\alpha}$, $v=e^{\eta\beta}$, and $w=e^{\eta\gamma}$ where $(\alpha,\beta,\gamma)$ is fixed and $\eta$ can vary in a small open interval containing zero. Hence, Equation~\eqref{eq:prob_gen} becomes
\begin{equation}\label{eq:mgf}
    \frac{Q_n(u,v,w)}{Q_n(1,1,1)}=\mathbb{E}\left(e^{\eta(\alpha X_n+\beta Y_n+\gamma Z_n)}\right).
\end{equation}
The right-hand side is the moment generating function of the random variable $\alpha X_n+\beta Y_n+\gamma Z_n$. From now on, let $\rho_0$ be the unique positive number such that $n=-f_x(\rho_0,0).$ Then, by Lemma~\ref{lem:1} (with $\epsilon= 0$) we have 
$
n\sim 2r\zeta(3)\rho_0^{-3}.
$ Moreover, if we write $u=e^{-a \rho}$, $v=e^{-b \rho}$, and $w=e^{-\gamma\rho}$ for $\rho>0$ as before, then we have $a=-\alpha\eta \rho^{-1}$, $b=-\beta\eta \rho^{-1}$, and $c=-\gamma\eta \rho^{-1}$. This implies that 
\[
|c|+|a|+r|b|=(|\gamma|+|\alpha|+r|\beta|)\eta\rho^{-1}.
\]
Now, if we choose $\eta$ and $\rho$ in such a way that $\eta\rho^{-1}=o(1)$ and $\rho\to 0$ as $n\to \infty$, then  by Lemma~\ref{lem:1} (with $\epsilon\to 0$), we get
\[
-f_x(\rho,\rho)\sim 2r\zeta(3)\rho^{-3}.
\]

Observe that it is possible to choose such $\eta$ and $\rho>0$ that satisfy $\eta=o(\rho_0)$ and  $\rho=o(\rho_0)$. In this case, the above asymptotic formula  implies that $-f_x(\rho,\rho)>-f_x(\rho_0,0)=n$ for large enough $n$. Similarly, we can also choose $\eta$ and $\rho>0$ such that $\eta=o(\rho_0)$, $\rho\to 0$, $\eta\rho^{-1}=o(1)$, and $\frac{\rho}{\rho_0}\to \infty$. This time, the asymptotic formula gives $-f_x(\rho,\rho)<n$. Hence, for $\eta=o(\rho_0)$ and $n$ large enough, the equation $n=-f_x(\rho,\rho)$ has a unique solution, which we denote by $\rho(\eta)$. Furthermore, it satisfies $\rho(\eta)\to 0$ and $\eta\rho(\eta)^{-1}=o(1)$ as $n\to \infty$. Therefore, we also have 
\[
n\sim  2r\zeta(3)\rho(\eta)^{-3}.
\]
The latter and the asymptotic estimate $n\sim 2r\zeta(3)\rho_0^{-3}$ yield $\rho(\eta)\sim \rho_0$ as $n\to \infty$ whenever $\eta=o(\rho_0)$.

From this point onward, we assume that $\eta=o(\rho_0).$ Since $\rho(\eta)$ and $\rho_0$ are asymptotically equivalent as $n\to \infty$, so are $f_{xx}(\rho(\eta), \rho(\eta))$ and $f_{xx}(\rho_0, 0).$ Thus, by \eqref{eq:asym_Q_n}, we have 
\begin{equation}\label{eq:mgf_asymp}
    \frac{Q_n(u,v,w)}{Q_n(1,1,1)}\sim \exp\Big(f(\rho(\eta),\rho(\eta))-f(\rho_0,0)+n(\rho(\eta)-\rho_0)\Big).
\end{equation}
We want to obtain a precise asymptotic estimate of the exponent of the right-hand side that holds uniformly for $\eta=o(\rho_0)$. We will use a Taylor approximation of the function $f(\rho(\eta),\rho(\eta))$ when $\eta$ is near $0$ (noting that the parameters $a$, $b$, and $c$ are themselves functions of $\eta$). So to highlight the variable $\eta$, we define
\begin{align*}
    g(x,y)
    & =\log F(e^{y\alpha},e^{y\beta},e^{y\gamma},e^{-x})\\
    &=-\sum_{l=1}^{r}\sum_{m=1}^{\infty}\sum_{k=1}^{m}\log\left(1-e^{y(\gamma+\alpha k+m\beta l)-x m}\right)\\
    &=\sum_{l=1}^{r}\sum_{m=1}^{\infty}\sum_{k=1}^{m}\Phi(xm-y(\gamma+\alpha k+m\beta l)).
\end{align*}
This is essentially the same as the function $f(x,y)$ (if $(a,b,c)$ is replaced by $(-\alpha, -\beta, -\gamma)$). However, in our case we have $a=-\alpha\eta \rho(\eta)^{-1}$, $b=-\beta\eta \rho(\eta)^{-1}$, and $c=-\gamma\eta \rho(\eta)^{-1}$. Hence, the meanings of the first and second variables will be different. For instance, we have the equation
\begin{equation}\label{eq:f2g}
    g(\rho(\eta),\eta)=f(\rho(\eta),\rho(\eta))\ \ \text{and}\ \ g(\rho_0,0)=f(\rho_0,0). 
\end{equation}
Similarly, the saddle-point equation $n=-f_x(\rho(\eta),\rho(\eta))$ becomes $n=-g_x(\rho(\eta),\eta).$
First, we apply implicit differentiation (with respect to $\eta$) to the latter equation, then by the mean value theorem, we get  
\[
\rho(\eta)-\rho_0= -\eta \, \frac{g_{xy}(\rho(\theta),\theta)}{g_{xx}(\rho(\theta),\theta)},
\]
for some real number $\theta$ between $0$ and $\eta.$ We can  estimate $g_{xy}(\rho(\theta),\theta)$ and $g_{xx}(\rho(\theta),\theta)$ via Mellin transform in the same way as in the proof of Lemma~\ref{lem:1}, but we need the Mellin transforms of the functions $g_{xx}(t,0)$ and $g_{xy}(t,0)$. The Mellin transform of $g_{xx}(t,0)$ is $r\zeta(s-3)\zeta(s-1)\Gamma(s)$. To determine the Mellin transform of $g_{x,y}(t,0)$, first we write
\begin{align*}
    g_{x,y}(t,0)
    & \,\,=\,\,-\sum_{l=1}^{r}\sum_{m=1}^{\infty}\sum_{k=1}^{m}m(\gamma+\alpha k+m\beta l)\Phi''(tm)\\
    & \,\,=\,\, -\sum_{m=1}^{\infty}m\left(rm\gamma+r\alpha\frac{m^2+m}{2}+m^2\beta\frac{r^2+r}{2}\right)\Phi''(tm)\\
    & \,\,=\,\, -\frac{r}{2}\sum_{m=1}^{\infty}\left((\alpha+(r+1)\beta) m^3+(\alpha+2\gamma)m^2\right)\Phi''(tm).
\end{align*}
Since the Mellin transform of $\Phi''(t)$ is $\zeta(s-1)\Gamma(s)$, the Mellin transform of $g_{x,y}(t,0)$ is 
\[
-\frac{r}{2}\left((\alpha+(r+1)\beta)\zeta(s-3) +(\alpha+2\gamma)\zeta(s-2)\right)\zeta(s-1)\Gamma(s).
\]
We deduce the following asymptotic formulas as $\rho\to0^{+}$:
\begin{align}
    g_{x,x}(\rho,0) & =6r\zeta(3)\rho^{-4}+\mathcal{O}(\rho^{-2})\label{eq:asy_gxx},\\
    g_{x,y}(\rho,0) 
    & =-3r\zeta(3)(\alpha+(r+1)\beta)\rho^{-4}-r\zeta(2)(\alpha+2\gamma)\rho^{-3}+\mathcal{O}(\rho^{-2})\label{eq:asy_gxy}.
\end{align}
The fact that $\rho(\theta)\sim \rho_0$  (since $|\theta|\leq |\eta|=o(\rho_0)$) and the argument in the proof of Lemma~\ref{lem:1} (with $\epsilon\to 0$) imply 
\begin{equation}\label{eq:dif_rho}
    \rho(\eta)-\rho_0\sim -\eta \, \frac{g_{xy}(\rho_0,0)}{g_{xx}(\rho_0,0)}=\mathcal{O}(|\eta|)=o(\rho_0). 
\end{equation} 
Similarly, we have the following estimate for the third partial derivatives
\[
\frac{\partial^k}{\partial x^k}\frac{\partial^l}{\partial y^l} g (x,y)\Big|_{(x,y)=(\rho(\theta),\theta)}=\mathcal{O}\Big(\rho_0^{-5}\Big),
\] 
uniformly for $\theta=o(\rho_0)$, and for any nonnegative integers $k$ and $l$ such that $k+l=3$. This shows that if $\eta=o(\rho_0)$, then we have the Taylor approximation
\begin{multline*}
g(\rho(\eta),\eta)  =  g(\rho_0,0)+g_{y}(\rho_0,0)\eta+g_x(\rho_0,0)(\rho(\eta)-\rho_0)\\
+ \frac{1}{2}\Big(g_{yy}(\rho_0,0) \, \eta^2+2g_{x y}(\rho_0,0) \,  \eta(\rho(\eta)-\rho_0)+g_{xx}(\rho_0,0) \, (\rho(\eta)-\rho_0)^2\Big)+\mathcal{O}(|\eta|^3\rho_0^{-5}),
\end{multline*}
where the implied constant in the error term is independent of $n$.
Using \eqref{eq:dif_rho} and saddle-point equation  $n=-g_x(\rho_0,0)$, we deduce that
\begin{multline*}
    g(\rho(\eta),\eta)-g(\rho_0,0)+n(\rho(\eta)-\rho_0)
    =\\
    g_y(\rho_0,0)\eta + \left(\frac{g_{yy}(\rho_0,0)g_{xx}(\rho_0,0)-(g_{xy}(\rho_0,0))^2}{g_{xx}(\rho_0,0)}\right)\frac{\eta^2}{2}+\mathcal{O}(|\eta|^3\rho_0^{-5}).
\end{multline*}
Let us define 
\begin{equation}\label{eq:def_mu_sigma}
    \sigma_n^2=\frac{g_{yy}(\rho_0,0)g_{xx}(\rho_0,0)-(g_{xy}(\rho_0,0))^2}{g_{xx}(\rho_0,0)} \, \text{ and }\,  \mu_n=g_y(\rho_0,0),
\end{equation}
and we choose $\eta=\frac{t}{\sigma_n}$ where $t$ is a fixed real number. Then our estimate \eqref{eq:mgf_asymp} becomes 
\begin{equation}\label{eq:quotient_Q}
    \frac{Q_n(u,v,w)}{Q_n(1,1,1)}\sim \exp\left(\frac{\mu_n}{\sigma_n}\ t+\frac{t^2}{2}+\mathcal{O}(\sigma_n^{-3}\rho_0^{-5})\right),
\end{equation}
as $n\to\infty$ (this  is valid as long as $\eta=\frac{t}{\sigma_n}=o(\rho_0)$, but we do not even know at this stage whether $\sigma_n^2$ is positive). Hence, let us    estimate $\sigma_n^2$. 

Assume first that $(\alpha,\beta)\neq (0,0)$. Once again, by the Mellin transform technique, we have
\begin{equation}\label{eq:asy_gyy}
    g_{yy}(\rho_0,0)  \sim r\left(2\alpha^2+3(r+1)\alpha\beta+(r+1)(2r+1)\beta^2\right)\zeta(3)\rho_0^{-4}.
\end{equation}
Putting the estimates \eqref{eq:asy_gxx}, \eqref{eq:asy_gxy} and \eqref{eq:asy_gyy} into the formula for $\sigma_n^2$ in \eqref{eq:def_mu_sigma}, we have 
\begin{equation}\label{eq:var_rho}
    \sigma_n^2\sim \left(\frac{r\zeta(3)}{2}\, \alpha^2+\frac{r(r^2-1)\zeta(3)}{2}\, \beta^2\right)\rho_0^{-4}.
\end{equation}
Hence, $\eta=\mathcal{O}(\rho_0^2)$ and the estimate \eqref{eq:quotient_Q} becomes 
\[
\frac{Q_n(u,v,w)}{Q_n(1,1,1)}\sim \exp\left(\frac{\mu_n}{\sigma_n}\ t+\frac{t^2}{2}+\mathcal{O}(\rho_0)\right).
\]
Therefore, if $(\alpha,\beta)\neq (0,0)$ and $t$ is any fixed real number, then the latter identity together with \eqref{eq:mgf} yield
\[
e^{-\mu_nt\sigma_n^{-1}}\mathbb{E}\left(e^{t\, \sigma_n^{-1}(\alpha X_n+\beta Y_n+\gamma Z_n)}\right)\sim e^{t^2/2},
\]
as $n\to \infty$. This means, by Curtiss' theorem~\cite{Curtiss}, that 
\[
\frac{\alpha X_n+\beta Y_n+\gamma Z_n-\mu_n}{\sigma_n}\overset{\mathrm{d}}{\to} \mathcal{N}(0,1) \ \ \text{as}\ \ n\to\infty.
\]
To obtain the asymptotic formulas for $\sigma_n^2$ and $\mu_n$, recall from the saddle-point equation and \eqref{eq:mellin} (with $\epsilon=0$) that $n= 2r\zeta(3)\rho_0^{-3}-\frac{1}{12}r\rho_0^{-1}+\mathcal{O}(1)$. Inverting this yields
\begin{equation}\label{eq:rhoton}
    \rho_0^{-1}=\frac{n^{1/3}}{(2r\zeta(3))^{1/3}}+\mathcal{O}(n^{-1/3}).
\end{equation}
So we first estimate $\sigma_n^2$ and $\mu_n$ in terms of $\rho_0$, then use the above to get the asymptotic formulas in terms of $n$.  Such an estimate for $\sigma_n^2$ is already given in \eqref{eq:var_rho}. The estimate of $\mu_n$ can be obtained easily from the Mellin transform of $g_y(t,0)$, which is
\[
\frac{r}{2}\left(\left(\alpha+(r+1)\beta\right)\zeta(s-2)+\left(\alpha+2\gamma\right)\zeta(s-1)\right)\zeta(s)\Gamma(s).
\]
By a straightforward calculation, we have 
\[
\sigma_n^2\sim \frac{\alpha^2+(r^2-1)\beta^2}{2^{7/3}(r\zeta(3))^{1/3}}\, n^{4/3}\, \text{ and }\, \mu_n= \left(\frac{1}{2}\alpha+\frac{r+1}{2} \beta\right)n+\frac{r^{1/3}\zeta(2)(\alpha+2\gamma)}{2^{5/3}(\zeta(3))^{2/3}}\, n^{2/3}+\mathcal{O}(n^{1/3}).
\]

If we now assume that $(\alpha,\beta)=(0,0)$ but $\gamma\neq 0$, then the Mellin transform of $g_{yy}(t,0)$ is $r\gamma^2\zeta(s-1)^2\Gamma(s)$ whose dominant singularity is a double pole at $s=2$. This leads to the asymptotic formula
$
g_{yy}(\rho_0,0)= r\gamma^2\rho_0^{-2}\log\left(\rho_0^{-1}\right)+\mathcal{O}(\rho_0^{-2}).
$
Hence, the formula in \eqref{eq:def_mu_sigma} gives
\begin{equation}\label{eq:var_gamma}
\sigma_n^2=r\gamma^2\rho_0^{-2}\log\left(\rho_0^{-1}\right)+\mathcal{O}(\rho_0^{-2}).    
\end{equation}
Thus, for a fixed real number $t$, we have 
\[
\eta=\frac{t}{\sigma_n}=\mathcal{O}(\rho_0|\log \rho_0|^{-1/2}).
\]
So we still have our desired condition that $\eta=o(\rho_0)$. Moreover, applying \eqref{eq:asy_gxx} and \eqref{eq:asy_gxy} with $(\alpha,\beta)=(0,0)$, the estimate \eqref{eq:dif_rho} becomes 
\[
\rho(\eta)-\rho_0\sim -\eta \, \frac{g_{xy}(\rho_0,0)}{g_{xx}(\rho_0,0)}=\mathcal{O}(|\eta| \rho_0)=\mathcal{O}(\rho_0^2|\log \rho_0|^{-1/2}).
\]
On the other hand, for any $\theta=o(\rho_0)$, we have the following: 
\begin{align*}
       g_{xxx}(\rho(\theta),\theta)&\,\,=\,\,\mathcal{O}(\rho_0^{-5}), &   g_{xxy}(\rho(\theta),\theta)\,\,=\,\,\mathcal{O}(\rho_0^{-4}),\\
       g_{xyy}(\rho(\theta),\theta)&\,\,=\,\,\mathcal{O}(\rho_0^{-3}|\log \rho_0|), &  g_{yyy}(\rho(\theta),\theta)\,\,=\,\,\mathcal{O}(\rho_0^{-3}).
\end{align*}
 Therefore, \eqref{eq:quotient_Q} becomes
 \[
 \frac{Q_n(u,v,w)}{Q_n(1,1,1)}\sim \exp\left(\frac{\mu_n}{\sigma_n}\ t+\frac{t^2}{2}+\mathcal{O}\left((\log n)^{-3/2}\right)\right).
 \]
Just as in the previous case, this is enough to prove the central limit theorem. The asymptotic formula for the variance in terms of $n$ can be obtained from \eqref{eq:var_gamma} using \eqref{eq:rhoton}. The proof of Proposition~\ref{propCLT} is complete.  

\bibliographystyle{amsplain-nodash}
\bibliography{bib}
\end{document}